\documentclass{article}%
\usepackage{amsmath}
\usepackage{amssymb}
\usepackage{euler}%
\usepackage{amsfonts}%
\usepackage{graphicx}
\providecommand{\U}[1]{\protect\rule{.1in}{.1in}}
\newtheorem{theorem}{Theorem}

\newtheorem{corollary}[theorem]{Corollary}

\newtheorem{lemma}[theorem]{Lemma}

\newtheorem{proposition}[theorem]{Proposition}

\newenvironment{proof}[1][Proof]{\noindent\textbf{#1.} }{\ \rule{0.5em}{0.5em}}
\begin{document}

\title{On the Constructive Theory of Jordan Curves}
\author{Douglas S. Bridges}
\maketitle

\begin{abstract}%
\noindent
Using a definition of \emph{Jordan curve} similar to that of Dieudonn\'{e}, we
prove that our notion is equivalent to that used by Berg et al. in their
constructive proof of the Jordan Curve Theorem \cite{JCTthm}. We then
establish a number of properties of Jordan curves and their corresponding
index functions, including the important Proposition \ref{jun04p1} and its
corollaries about lines crossing a Jordan curve at a smooth point. The final
section is dedicated to proving that the index of a point with respect to a
piecewise smooth Jordan curve in the complex plane is identical to the
familiar winding number of the curve around that point. The paper is written
within the framework of Bishop's constructive analysis. Although the work in
Sections 3--5 is almost entirely new, the paper contains a substantial amount
of expository material for the benefit of the reader.

\end{abstract}%

\normalfont\sf

\section{Introduction}

This paper delves into the constructive theory of Jordan curves, beyond the
fundamental paper \cite{JCTthm}. Our setting is Bishop's constructive
mathematics \textbf{(BISH)}---very roughly, mathematics carried out with
intuitionistic/constructive logic and an appropriate set- or type-theory---and
assumes that the reader is either familiar with Bishop's approach or else has
access to, for example, one of \cite{Bishop,BB,Handbook,BVtech}. For
background in complex analysis we refer the reader to Chapter 5 of either
\cite{Bishop} or \cite{BB},\footnote{%
\normalfont\sf
We shall use (A.B.C) to refer to item (B.C) in Chapter A of \cite{Bishop},
\cite{BB}, or \cite{BVtech}.} and to our improved version \cite[Theorem
1]{Bridges24} of Cauchy's integral theorem (and hence of all Bishop's results
that refer to that theorem).\footnote{%
\normalfont\sf
What is the improvement in question? It is that, where Bishop uses restrictive
notions of path connectedness and homotopy, in \cite{Bridges24} we use the
standard notions.}

In \cite{JCTthm}, Berg et al. introduced a Euclidean-geometrical notion of
\emph{index}, a nonnegative integer $\mathsf{ind}(\zeta;\gamma)$%
\emph{\ }that\emph{\ }intuitively measures the number of times a closed path
$\gamma\ $in the complex plane $\mathbb{C}$ winds around a point $\zeta$
bounded away from $\gamma$. Defining a \emph{Jordan curve\ }to be what, in
classical terms, would be a homeomorphism of the unit circle in $\mathbb{C}$
onto a compact---that is, complete, totally bounded---subset of $\mathbb{C}$,
and using an ingenious tesselation argument, they then proved the
\textbf{constructive Jordan Curve Theorem},\footnote{%
\normalfont\sf
Actually, Berg et al. apply the name \emph{Jordan Curve Theorem} to a slightly
weaker result stated in the introduction to their paper.}

\begin{quote}
\textbf{JCT: \ \ }\emph{If }$\gamma$\emph{\ is a Jordan curve, and }%
$a,b$\emph{\ are points off }$\gamma$\emph{, then either }$a$\emph{\ and~}%
$b$\emph{\ can be joined by a polygonal path which is bounded away from
}$\gamma$\emph{\ or else the indexes of }$a$\emph{\ and }$b$\emph{\ differ by
}$1\ $\cite[Theorem 1]{JCTthm}.
\end{quote}%

\noindent
In our present paper we adopt a different definition of Jordan curve, one
based on that given by Dieudonn\'{e} \cite[pp. 255--256]{Dieudonne}. We prove
that our definition is more-or-less equivalent to that in \cite{JCTthm},
before establishing, constructively, a number of results---in particular ones
dealing with lines that intersect a Jordan curve---that enable us to prove,
finally, that the index of a point $\zeta$ with respect to a piecewise smooth
Jordan curve equals the winding number of $\gamma$ about $\zeta$ as defined in
complex analysis. For the reader's benefit and for expository clarity, we
include in our presentation some possibly familiar definitions and elementary
results; but most of the work below is completely new constructive analysis.

\section{Preliminaries}

In this section we lay down some preliminary definitions and facts many of
which may be familiar to the reader (though some of our definitions differ
from those in \cite{Bishop,BB})

Let $\mathbb{F}$ denote either the real line $\mathbb{R}$ or the complex plane
$\mathbb{C}$. When $z_{1},z_{2}$ are distinct points of $\mathbb{C}$ we use
these notations\footnote{%
\normalfont\sf
It should be clear from the context whether $\left(  z,z^{\prime}\right)  $
denotes an ordered pair of complex numbers or the open segment in $\mathbb{C}
$ with endpoints $z$ and $z^{\prime}$.} for line segments in $\mathbb{C\ }%
$with endpoints $z$ and $z^{\prime}$:%
\begin{align*}
\left[  z,z^{\prime}\right]   &  \equiv\{(1-t)z+tz^{\prime}:0\leq t\leq1\}\\
(z,z^{\prime}] &  \equiv\{(1-t)z+tz^{\prime}:0<t\leq1\},\\
(z,z^{\prime})\, &  \equiv\{(1-t)z+tz^{\prime}:0<t<1\},\\
\lbrack z,z^{\prime}) &  \equiv\{(1-t)z+tz^{\prime}:0\leq t<1\}.
\end{align*}

For each located subset $K$ of $\mathbb{F}$ and each $r>0$ we define%
\[
K_{r}\equiv\{z\in\mathbb{F}:\rho(z,K)\leq r\},
\]
where the \textbf{distance from }$z$\textbf{\ to} $K,$%
\[
\rho(z,K)\equiv\inf\left\{  \left\vert z-\zeta\right\vert :\zeta\in K\right\}
\text{,}%
\]
exists, by definition of \emph{located}. If $U$ is an open subset $\mathbb{F}
$ such that $K_{r}\subset U$ for some $r>0$, we say that $K$ is \textbf{well
contained }in\textbf{\ }$U$, which property we denote by $K\subset\subset U$.
If $K$ is totally bounded (and hence located), then $K_{r}$ is compact for
each $r>0$.

We recall here

\begin{quote}
\textbf{Bishop's Lemma:} \emph{Let }$A$\emph{\ be a complete, located subset
of a metric space }$\left(  X,\rho\right)  $\emph{. Then for each }$x\in
X$\emph{\ there exists }$a\in A$\emph{\ such that if\footnote{%
\normalfont\sf
It is important to note that in a metric space, `$x\neq a$' means
`$\rho(x,a)>0$', not `$\lnot(x=a)$': see pages 10--11 of \cite{BVtech}.}
}$x\neq a $\emph{, then }$\rho(x,A)>0$\emph{\ }(\cite[p. 177, Lemma 7]%
{Bishop},\cite[(3.1.1)]{BVtech}).
\end{quote}

Here are two simple, but useful, consequences of Bishop's Lemma.

\begin{lemma}
\label{jan01l1}Let $L$ be a line in $\mathbb{C}$, $\zeta_{0}\in L$, and
$\zeta\in\mathbb{C}-\left\{  \zeta_{0}\right\}  $. Then there exists
$\zeta^{\prime}\in L$ with $\left\vert \zeta^{\prime}-\zeta_{0}\right\vert
\leq3\left\vert \zeta-\zeta_{0}\right\vert $ such that if $\zeta\neq
\zeta^{\prime}$, then $\rho(\zeta,L)>0$.
\end{lemma}

\begin{proof}
Let%
\[
A=\left\{  z\in L:\left\vert z-\zeta_{0}\right\vert \leq3\left\vert
\zeta-\zeta_{0}\right\vert \right\}  .
\]
which, being compact, is complete and located. If $z\in L$, then either $z\in
A$ or $\left\vert z-\zeta_{0}\right\vert \geq2\left\vert \zeta-\zeta
_{0}\right\vert $. Since in the latter case%
\[
\left\vert z-\zeta\right\vert \geq\left\vert z-\zeta_{0}\right\vert
-\left\vert \zeta-\zeta_{0}\right\vert >\left\vert \zeta-\zeta_{0}\right\vert
\geq\rho(\zeta,A)\text{,}%
\]
we see that $\rho(\zeta,L)$ exists and equals $\rho(\zeta,A)$. The proof is
completed by applying Bishop's Lemma to $A$.%
\hfill

\end{proof}

\begin{lemma}
\label{nov12l-1}If $K$ is a compact subset of the half open interval $[a,b)$
\emph{(}respectively, $(a,b]$\emph{)} in $\mathbb{R}$, then there exists $r>0
$ such that $K\subset\lbrack a,b-r]$ \emph{(}respectively, $K\subset\lbrack
a+r,b]$\emph{)}. If $K$ is a compact subinterval of $\left(  a,b\right)  $,
then $K\subset\subset\left(  a,b\right)  $.\emph{\footnote{\textsf{Things are
different constructively in }$\mathbb{C}$\textsf{: there exists a recursive
example of a compact subset }$K$\textsf{\ of the open unit ball }%
$B(0,1)\ $\textsf{in }$\mathbb{C}$\textsf{\ such that there are points of }%
$K$\textsf{\ arbitrarily close to the boundary of }$B(0,1)$\textsf{. See
}\cite[Thm (3.1)]{BillFred84}.}}
\end{lemma}

\begin{proof}
If $K\subset\lbrack a,b)$ is compact, then $b\neq x$ for each $x\in K$. Since
$K$ is complete and located in $\mathbb{R}$, Bishop's Lemma shows that
$0<r\equiv\rho(b,K)$, so $K\subset\lbrack a,b-r]$. The case $K\subset(a,b]$ is
handled similarly; the case $K\subset\left(  a,b\right)  $ then readily
follows.%
\hfill

\end{proof}%

\medskip

We say that a complex-valued mapping $f$ whose domain $D$ is a subset of
\thinspace$\mathbb{F}$ is

\begin{itemize}
\item \textbf{continuous}

\begin{itemize}
\item[--] \textbf{on the compact set }$K\subset\mathbb{F}\ $if $f$ is defined
and uniformly continuous on $K$;

\item[--] \textbf{on the open set }$U\subset\mathbb{F}$ if it is defined on
$U$ and continuous on each compact set well contained in $U$.\footnote{%
\normalfont\sf
We shall make good use of the fact that if $f$ is continuous on an open set
$U$ and $S$ is a totally bounded set that is well contained in $U$, then $f$
is uniformly continuous on $S$.}
\end{itemize}

\item \textbf{differentiable}

\begin{itemize}
\item[--] \textbf{on the compact set} $K\subset\mathbb{F}$ if there exists a
continuous mapping $g:K\rightarrow\mathbb{C}$ with the property that for each
$\varepsilon>0$ there exists $\delta>0$ such that
\begin{equation}
\forall\zeta,z\in K\,(\left\vert z-\zeta\right\vert \leq\delta\Rightarrow
\left\vert f(z)-f(\zeta)-g(\zeta)(z-\zeta)\right\vert \leq\varepsilon
\left\vert z-\zeta\right\vert );\label{no1}%
\end{equation}

\item[--] \textbf{on the open set }$U\subset\mathbb{F}$, if there exists a
continuous mapping $g:U\rightarrow\mathbb{C}$ such that for each compact
$K\subset\subset U$, $f$ is differentiable and (\ref{no1}) holds.
\end{itemize}
\end{itemize}%

\noindent
In each case of the definition of differentiability, $g$ is called the
\textbf{derivative }of $f$ on the subset of $\mathbb{C}$ under consideration.

By a \textbf{path}\footnote{%
\normalfont\sf
Bishop requires paths to be piecewise differentiable.}\textbf{\ }in
$\mathbb{C}$ with \textbf{endpoints} $z_{0}$ and $z_{1}$ we mean a continuous,
complex-valued function $\gamma$ on a proper compact interval $\left[
a,b\right]  $---the \textbf{parameter interval} of $\gamma$---such that
$\gamma(a)=z_{0}$ and $\gamma(b)=z_{1}$; in that case, the path is said to
\textbf{join} $z_{0}$ and $z_{1}$, and the values of the function $\gamma$ are
the \textbf{points of} $\gamma$. We call $\gamma\ $a \textbf{closed path} if
$\gamma(a)=\gamma(b)$.

We say that $\gamma$ is \textbf{piecewise differentiable }if there exist real
numbers $t_{0}=a<t_{1}<\cdots<t_{n}=b$ such that $\gamma$ is
differentiable\footnote{%
\normalfont\sf
This means that the one-sided derivatives exist at $t_{k}$, but does not
require those two derivatives to be equal.} on each interval $\left[
t_{k},t_{k+1}\right]  $. In that case, if $t_{k}<t<t_{k+1}$ and $\gamma
^{\prime}(t)\neq0$, then $\gamma\,^{\prime}(t)$ is a tangent vector to
$\gamma$ at $t$, $\gamma^{\prime}(s)\neq0$ for all $s$ sufficiently close to
$t$ (since $\gamma$ is continuously differentiable on $\left[  t_{k}%
,t_{k+1}\right]  $), and we call $\gamma(t)$ a \textbf{smooth point} of
$\gamma$. If for each $k<n$ every point of $(t_{k},t_{k+1})$ is smooth, then
$\gamma$ is a \textbf{piecewise smooth path}. In the special case where
$a=t_{0}<t_{1}=b$, we drop `piecewise' from these definitions.

Let $\gamma_{1}:\left[  a_{1},b_{1}\right]  \rightarrow\mathbb{C}$ and
$\gamma_{2}:\left[  a_{2},b_{2}\right]  \rightarrow\mathbb{C}$ be paths such
that $\gamma_{1}(b_{1})=\gamma_{2}(a_{2})$. Let $\lambda$ be the continuous,
strictly increasing map of $[b_{1},b_{1}+b_{2}-a_{2}]$ onto $\left[
a_{2},b_{2}\right]  $ defined by%
\[
\lambda(b_{1}+t(b_{2}-a_{2}))\equiv(1-t)a_{2}+tb_{2}\ \ \ \ \left(  0\leq
t\leq1\right)  .
\]
Then there is a unique (uniformly) continuous function $\gamma_{3}$ on
$\left[  b_{1},b_{1}+b_{2}-a_{2}\right]  $ such that%
\[
\gamma_{3}(t)=\left\{
\begin{array}
[c]{ll}%
\gamma_{1}(t) & \text{if\ }t\in\left[  a_{1},b_{1}\right]  \text{ }\\
\gamma_{2}(\lambda(t)) & \text{if }t\in\left[  b_{1},b_{1}+b_{2}-a_{2}\right]
.
\end{array}
\right.
\]
This function $\gamma_{3}$ is called the \textbf{sum of the paths} $\gamma
_{1}$ and $\gamma_{2}$ and is denoted by $\gamma_{1}+\gamma_{2}$. The
\textbf{sum }$\gamma_{1}+\cdots+\gamma_{n}$\textbf{\ of }$n>1$\textbf{\ paths
}is defined inductively by $\gamma\equiv(\gamma_{1}+\cdots+\gamma
_{n-1})+\gamma_{n}$.

The \textbf{linear path }joining $z_{0}$ and $z_{1}$ is the path%
\[
\mathsf{lin}(z_{0},z_{1}):t\rightsquigarrow(1-t)z_{0}+tz_{1}\ \ \ (0\leq
t\leq1).
\]
A path $\gamma$ is \textbf{polygonal}, or \textbf{piecewise linear}, if there
exist $z_{1},\ldots,z_{n}$, called the \textbf{vertices }of $\gamma$, such
that%
\[
\gamma=\mathsf{lin}(z_{1},z_{2})+\mathsf{lin}(z_{2},z_{3})+\cdots
+\mathsf{lin}(z_{n-1},z_{n}).
\]

If $\gamma$ is a path with parameter interval $I$, then $\gamma(I)$, being the
image of a compact interval under a (uniformly) continuous function, is
totally bounded; so its closure $\mathsf{car}(\gamma)\equiv\overline
{\gamma(I)}$, the \textbf{carrier} of $\gamma$, is compact in $\mathbb{C}$ and
hence located. A subset $X$ of $\mathbb{C}$ is \textbf{bounded away} from
$\gamma$ by $c>0$ if $\rho(z,\mathsf{car}(\gamma))\geq c$ for all $z\in X$. If
$X=\left\{  \zeta\right\}  $ is a singleton, then it is bounded away from
$\gamma$ precisely when $\zeta$ belongs to the \textbf{metric complement }of
$\mathsf{car}(\gamma)$,%
\[
-\mathsf{car}(\gamma)\equiv\left\{  z\in\mathbb{C}:\rho(z,\mathsf{car}%
(\gamma))>0\right\}  ,
\]
in which case \textbf{the point }$\zeta$\textbf{\ is} \textbf{bounded away
from}, or \textbf{off},\ the path $\gamma$. We say that $\gamma$ \textbf{lies
in, }or \textbf{is a\ path in}, the set $S\subset\mathbb{C}$ if either $S$ is
compact and $\mathsf{car}(\gamma)\subset S$ or else $S$ is open and
$\mathsf{car}(\gamma)\subset\subset S$.

\begin{proposition}
\label{jun06p1}Let $\gamma$ be a path in $\mathbb{C}$ with parameter interval
$I$. Then $\zeta\in\mathbb{C}$ is bounded away from $\gamma$ if and only if
$\zeta\neq z\ $for each $z\in\mathsf{car}(\gamma)$.
\end{proposition}

\begin{proof}
The proof of\ `only if' is trivial. Since $\mathsf{car}(\gamma)$ is compact,
it is complete and located, so, by Bishop's Lemma, there exists $z\in
\mathsf{car}(\gamma)$ such that if $\zeta\neq z$, then $\rho(\zeta
,\mathsf{car}(\gamma))>0$, from which `if' follows immediately.%
\hfill

\end{proof}%

\medskip
We will need this simple result, whose proof we include for the sake of completeness.

\begin{proposition}
\label{jun10p3}Let $a<b$, and let $f$ be a continuous, strictly increasing
mapping of $\left[  a,b\right]  $ into $\left[  f(a),f(b)\right]  $. Then the
range of $f$ is $\left[  f(a),f(b)\right]  $, and $f^{-1}$ is (uniformly)
continuous on $\left[  f(a),f(b)\right]  $.
\end{proposition}

\begin{proof}
It follows from the remark on pages 40-41 of \cite{BB} that $f$ maps $\left[
a,b\right]  $ onto $\left[  f(a),f(b)\right]  $. Fixing $\varepsilon>0$,
choose $N\geq3$ and points $t_{k}$ such that $a=t_{0}<t_{1}<\cdots<t_{N}=b$
and $t_{k+3}-t_{k}<\varepsilon\ (0\leq k\leq N-3)$. Let $s_{k}=f(t_{k})$ for
each $k$. Then $f(a)=s_{0}<s_{1}<\cdots<s_{N}=f(b)$, so%
\[
\delta\equiv\tfrac{1}{2}\min\left\{  s_{k+1}-s_{k}:0\leq k<N\right\}  >0.
\]
Consider $s,s^{\prime}\in\left[  f(a),f(b)\right]  $ with $\left\vert
s-s^{\prime}\right\vert <\delta$. Either $\left\vert f^{-1}(s)-f^{-1}%
(s^{\prime})\right\vert <\varepsilon$ or else $0<\left\vert f^{-1}%
(s)-f^{-1}(s^{\prime})\right\vert $. In the latter case, we may assume without
loss of generality that $f^{-1}(s)\,<f^{-1}(s^{\prime})$ and therefore
$s<s^{\prime}$. Choose $k$ such that $s_{k}\leq s<s_{k+2}$. If $k\leq N-3$,
then $s_{k}\leq s<s^{\prime}<s_{k+2}+\delta<s_{k+3}$, so $t_{k}\leq
f^{-1}(s)<f^{-1}(s^{\prime})<t_{k+3}$ and therefore $0<f^{-1}(s^{\prime
})-f^{-1}(s)<\varepsilon$. If $k=N-2$, then a similar, but simpler, argument
completes the proof.%
\hfill

\end{proof}%

\medskip
Let $\gamma$ be a path in $\mathbb{C}$ with parameter interval $\left[
a,b\right]  $, and let $\left[  c,d\right]  $ be a proper compact interval in
$\mathbb{R}$. Let $f$ be the continuous, strictly increasing mapping of
$\left[  a,b\right]  $ into $\left[  c,d\right]  $ defined by%
\[
f(\left(  1-t)a+tb\right)  =\left(  1-t\right)  c+td\text{ \ }\left(  0\leq
t\leq1\right)  .
\]
From Proposition \ref{jun10p3} we see that $f$ maps $\left[  a,b\right]  $
onto $\left[  c,d\right]  $ and that the inverse function $f^{-1}$ is
continuous on $\left[  c,d\right]  $. Hence $\gamma\circ f^{-1}$ is a path
with parameter interval $\left[  c,d\right]  $, and $\mathsf{car}(\gamma\circ
f^{-1})=\mathsf{car}(\gamma)$. We call $\gamma$ and $\gamma\circ f^{-1}$
\textbf{equivalent paths}. Equivalence of paths is an equivalence relation.

Closed paths $\gamma_{0},\gamma_{1}$ that lie in a compact set $K\subset
\mathbb{C}$ and have the same parameter interval $\left[  a,b\right]  $ are
\textbf{homotopic in }$K$ if there exists a continuous function $\sigma
:[0,1]\times\left[  a,b\right]  \rightarrow K$ such that for each $t\in\left[
0,1\right]  $, the function $\sigma_{t}:\left[  a,b\right]  \rightarrow
\mathbb{C}$ defined by $\sigma_{t}(x)\equiv\sigma(t,x)$ is a\footnote{%
\normalfont\sf
Bishop requires that the paths $\sigma_{t}$ $(0\leq t\leq1)\ $be piecewise
differentiable \cite[(5.3.11)]{BB}.} closed path, $\sigma_{0}=\gamma_{0}$, and
$\sigma_{1}=\gamma_{1}$. The function $\sigma$ is then called a
\textbf{homotopy\ }of $\gamma_{0}$ and $\gamma_{1}$. If also $U\subset
\mathbb{C}$ is open and $K\subset\subset U$, then $\gamma_{0}$ and $\gamma
_{1}$ are said to be \textbf{homotopic in }$U$. If $\gamma_{1}$ is a constant
path, then $\gamma_{0}$ is said to be \textbf{null-homotopic}.

An open set $U\subset\mathbb{C}$ is

\begin{itemize}
\item \textbf{path connected }if any two points of $U$ can be joined by a path
in $U$;

\item \textbf{connected }if for any two inhabited open subsets $A,B\ $of $U$
with $U=A\cup B$ there exists $z\in$ $A\cap B$;\footnote{%
\normalfont\sf
Bishop uses the term \emph{connected }to mean that any two points of $U$ are
joined by a piecewise differentiable path in $U$. In the definition of
\emph{simply connected}~he uses \emph{connected} (in his sense) where we have
\emph{path connected.}}

\item \textbf{simply connected }if it is path connected and every closed path
in $U$ is null-homotopic.
\end{itemize}%

\noindent
Note that path connected implies connected and that \emph{classically}, in the
case of an open subset $U$ of $\mathbb{C}$, connected implies path connected
\cite[(9.7.2)]{Dieudonne}.

\section{Jordan curves}

With one eye on Dieudonn\'{e}'s definition, \cite[pp. 255--256]{Dieudonne}, we
define a \textbf{Jordan curve}, or \textbf{simple closed curve}, to be a
closed path $\gamma:\left[  a,b\right]  \rightarrow\mathbb{C}$ such that

\begin{enumerate}
\item[\textsf{J1}] if $\left(  t,t^{\prime}\right)  \in\lbrack a,b]\times
(a,b)$ and $t\neq t^{\prime}$, then $\gamma(t)\neq\gamma(t^{\prime})$;

\item[\textsf{J2}] for each compact $K\subset(a,b)$, $\gamma^{-1}$ is
continuous on $\gamma(K)$.
\end{enumerate}

In \cite{JCTthm}, Berg et al. use definitions of \emph{closed path} and
\emph{Jordan curve} different from ours. For them, a closed path is an ordered
pair $J=(S,f)$ where $S\subset\mathbb{C}$ and $f$ is a (uniformly) continuous
mapping $\gamma$ of the \textbf{unit circle}
\[
\mathsf{U}\equiv\left\{  z\in\mathbb{C}:\left\vert z\right\vert =1\right\}
\]
onto $S$; for $J$ to be a Jordan curve they require also that $f$ be injective
and have uniformly continuous inverse.\footnote{%
\normalfont\sf
They also impose the requirement of moduli of continuity for $f$ and $f^{-1}%
$.} We shall refer to these as `Berg notions'. Our main aim in this section is
to reveal the connection between the Berg notion of Jordan curve and our one.

\begin{proposition}
\label{jun12p1}Let $\gamma:[a,b]\rightarrow\mathbb{C}$ be
pointwise\footnote{\textsf{Pointwise continuity is continuity at each point,
in the usual sense. }} continuous at $a$ and $b$, with $\gamma(a)=\gamma(b)$.
Then the following conditions are equivalent:

\begin{enumerate}
\item[\emph{(i)}] for all $\left(  t,t^{\prime}\right)  \in\lbrack
a,b]\times(a,b)$, if $t\neq t^{\prime}$, then $\gamma(t)\neq\gamma(t^{\prime
})$;

\item[\emph{(ii)}] the restriction of $\gamma$ to $[a,b)$ is injective;

\item[\emph{(iii)}] the restriction of $\gamma$ to $(a,b]$ is injective.
\end{enumerate}
\end{proposition}

\begin{proof}
If (i) holds and $t,t^{\prime}$ are points of$\ [a,b)$ with $t<t^{\prime}$,
then $a<t^{\prime}$, so $\left(  t,t^{\prime}\right)  \in\left[  a,b\right]
\times(a,b)$ and therefore $\gamma(t)\neq\gamma(t^{\prime})$. \ Hence (i)
implies (ii). Now assume (ii), and let $a<t,t^{\prime}\leq b$ and $t\neq
t^{\prime}$. Either $t\neq b$ or $t^{\prime}\neq b$; to illustrate, we take
the former case. Since $a,t\in\lbrack a,b)$ and $a\neq t$, $\gamma
(a)\neq\gamma(t)$. By the pointwise continuity of $\gamma$ at $a$ and $b$,
there exists $\delta$ with $0<\delta<b-a$ such that if $s\in\left[
a,b\right]  $ and either $\left\vert s-a\right\vert <\delta$ or $\left\vert
s-b\right\vert <\delta$, then $\left\vert \gamma(s)-\gamma(a)\right\vert
=\left\vert \gamma(s)-\gamma(b)\right\vert <\left\vert \gamma(t)-\gamma
(a)\right\vert $. Either $t^{\prime}<b$ or else $t^{\prime}>b-\delta$. In the
first case, $t,t^{\prime}\in\lbrack a,b)$, so by (ii), $\gamma(t)\neq
\gamma(t^{\prime})$. In the second case, $\left\vert t^{\prime}-b\right\vert
<\delta$, so $\left\vert \gamma(t^{\prime})-\gamma(a)\right\vert <\left\vert
\gamma(t)-\gamma(a)\right\vert $ and therefore $\gamma(t)\neq\gamma(t^{\prime
})$. It follows that (ii) implies (iii); a similar argument proves that (iii)
implies (ii). Finally, assume that (iii) and hence (ii) holds. If $\left(
t,t^{\prime}\right)  \in\lbrack a,b]\times(a,b)$ and $t\neq t^{\prime}$, then
either $a<t$, in which case $t,t^{\prime}\in(a,b]$ and therefore
$\gamma(t)\neq\gamma(t^{\prime})$, or else $t<b$, in which case $t,t^{\prime
}\in\lbrack a,b)$ and again $\gamma(t)\neq\gamma(t^{\prime})$.%
\hfill

\end{proof}%

\medskip
We state here, without the simple proof, the following:

\begin{lemma}
\label{dec23p1}Let $X,Y$ be metric spaces, $f$ a uniformly continuous
injective mapping of $X$ onto $Y$ such that $f^{-1}:Y\rightarrow X$ is
uniformly continuous. Then

\begin{enumerate}
\item[\emph{(i)}] $\left(  x_{n}\right)  _{n\geq1}$ is a Cauchy sequence in $X
$ if and only if $\left(  f(x_{n})\right)  _{n\geq1}$ is a Cauchy sequence in
$Y$;

\item[\emph{(ii)}] $X$ is complete if and only if $Y$ is complete; and

\item[\emph{(iii)}] $X$ is compact if and only if $Y$ is compact.
\end{enumerate}
\end{lemma}

\begin{proposition}
\label{jan30p1}Let $\gamma:\left[  a,b\right]  \rightarrow\mathbb{C}$ be a
Jordan curve and $K$ a subset of $\left(  a,b\right)  $. Then $\gamma(K)$ is
compact if and only if $K$ is compact.
\end{proposition}

\begin{proof}
Suppose that $\gamma(K)$ is compact. By \textsf{J1}, $\gamma(t)\neq\gamma(a)$
and $\gamma(t)\neq\gamma(b)$ for each $t\in K$; whence, by Bishop's Lemma,
\[
0<m\equiv\min\left\{  \rho(\gamma(a),\gamma(K),\rho(\gamma(b),\gamma
(K)\right\}  .
\]
Since $\gamma$ is uniformly continuous on $\left[  a,b\right]  $, there exists
$\delta>0$ such that if $t,t^{\prime}\in\left[  a,b\right]  $ and $\left\vert
t-t^{\prime}\right\vert <\delta$, then $\left\vert \gamma(t)-\gamma(t^{\prime
})\right\vert <m$. It follows that $\left\vert a-\gamma^{-1}(z)\right\vert
\geq\delta$ and $\left\vert b-\gamma^{-1}(z)\right\vert \geq\delta$ for each
$z\in K$; whence $\gamma^{-1}(K)\subset\left[  a+\delta,b-\delta\right]
\subset(a,b)$. By \textsf{J2}, $\gamma^{-1}$ is uniformly continuous on
$\gamma\lbrack a+\delta,b-\delta]$ and hence on $\gamma(K)$. Applying Lemma
\ref{dec23p1} with $f=\gamma^{-1}$ and $X=\gamma(K)$, we now see that $K$ is
compact. Conversely, if $K$ is compact, then by \textsf{J}2, $\gamma^{-1}$ is
uniformly continuous on $\gamma(K)$, so we can apply Lemma \ref{dec23p1} with
$f=\gamma$ and $X=K$, to prove that $\gamma(K)$ is compact.%
\hfill

\end{proof}

\begin{lemma}
\label{dec23p2}Let $\gamma:\left[  a,b\right]  \rightarrow\mathbb{C}$ be a
Jordan curve. Then for each $\varepsilon>0$ there exists $\delta>0$ such that
if $t\in\left[  a,b\right]  $ and $\left\vert \gamma(t)-\gamma(a)\right\vert
<\delta$, then either $a\leq t<a+\varepsilon$ or $b-\varepsilon<t\leq b$.
\end{lemma}

\begin{proof}
Given $\varepsilon\ $with $0<\varepsilon<\frac{1}{2}(b-a)$, let $K\equiv
\left[  a+\varepsilon/2,b-\varepsilon/2\right]  $. Then $K\subset\left(
a,b\right)  $ and is compact, so by the previous lemma, $\gamma(K)\ $is
compact. Since $\gamma(a)\neq\gamma(t)$ for each $t\in K$, it follows from
Bishop's Lemma that $\delta\equiv\rho(\gamma(a),\gamma(K))>0$. If $t\in\left[
a,b\right]  $ and $\left\vert \gamma(t)-\gamma(a)\right\vert <\delta$, then
$t\notin K$, so either $t\in\lbrack a,a+\varepsilon)$ or else $t\in
(b-\varepsilon,b]$.%
\hfill

\end{proof}

\begin{proposition}
\label{nov22p2}Let $\gamma:[a,b]\rightarrow\mathbb{C}$ be a closed path
satisfying \textsf{J1}. Then the following conditions are equivalent.

\begin{enumerate}
\item[\emph{(i)}] $\gamma$ is a Jordan curve;

\item[\emph{(ii)}] $\gamma^{-1}:[a,b)\rightarrow\mathbb{C}$ is uniformly
continuous on $\gamma\lbrack a,c]$ for each $c\in(a,b)$;

\item[\emph{(iii)}] $\gamma^{-1}:(a,b]\rightarrow\mathbb{C}$ is uniformly
continuous on $\gamma\lbrack c,b]$ for each $c\in(a,b)$.
\end{enumerate}%

\noindent
If (i) holds and $K$ is a compact subset of either $[a,b)$ or $(a,b]$, then
$\gamma(K)$ is compact.
\end{proposition}

\begin{proof}
Suppose that (i) holds. Consider any $c\in(a,b)$ and any $\varepsilon$ with
$0<\varepsilon<\min\left\{  c-a,b-c\right\}  $. Using Lemma
\textsf{\ref{dec23p2},} choose $\delta_{1}$ with $0<\delta_{1}<\varepsilon/3$
such that if $t\in\left[  a,b\right]  $ and $\left\vert \gamma(t)-\gamma
(a)\right\vert <3\delta_{1}$, then either $a\leq t<a+\varepsilon$ or
$b-\varepsilon<t\leq b$. Note that if $t\in\left[  a,c\right]  $ then
$t<b-\varepsilon$, so if also $\left\vert \gamma(t)-\gamma(a)\right\vert
<3\delta_{1}$, then $a\leq t<a+\varepsilon$. By the continuity of $\gamma$,
there exists $\delta_{2}$ with $0<\delta_{2}<\delta_{1}$ such that if
$t,t^{\prime}\in\left[  a,b\right]  $ and $\left\vert t-t^{\prime}\right\vert
<\delta_{2}$, then $\left\vert \gamma(t)-\gamma(t^{\prime})\right\vert
<\delta_{1}$. By \textsf{J2}, there exists $\delta_{3}\ $with $0<\delta
_{3}<\delta_{2}$ such that if $t,t^{\prime}\in\left[  a+\delta_{2},c\right]
$, and $\left\vert \gamma(t)-\gamma(t^{\prime})\right\vert <\delta_{3}$, then
$\left\vert t-t^{\prime}\right\vert <\varepsilon$. Now consider $t,t^{\prime
}\in\left[  a,c\right]  $ with $\left\vert \gamma(t)-\gamma(t^{\prime
})\right\vert <\delta_{3}$. We have three alternatives: $\left\vert
\gamma(t)-\gamma(a)\right\vert <2\delta_{1}$, or $\left\vert \gamma(t^{\prime
})-\gamma(a)\right\vert <2\delta_{1}$, or both $\left\vert \gamma
(t)-\gamma(a)\right\vert >\delta_{1}$ and $\left\vert \gamma(t^{\prime
})-\gamma(a)\right\vert >\delta_{1}$. In the first case we have $a\leq
t<a+\varepsilon$; moreover,
\[
\left\vert \gamma(t^{\prime})-\gamma(a)\right\vert \leq\left\vert
\gamma(t^{\prime})-\gamma(t)\right\vert +\left\vert \gamma(t)-\gamma
(a)\right\vert <\delta_{3}+2\delta_{1}<3\delta_{1},
\]
so $a\leq t^{\prime}<a+\varepsilon$ and therefore $\left\vert t-t^{\prime
}\right\vert <\varepsilon$. Similarly, in the second case we have $\left\vert
t-t^{\prime}\right\vert <\varepsilon$. In the third case we have $\left\vert
t-a\right\vert \geq\delta_{2}$ and $\left\vert t^{\prime}-a\right\vert
\geq\delta_{2}$, so $t,t^{\prime}\in\left[  a+\delta_{2},c\right]  $ and
$\left\vert \gamma(t)-\gamma(t^{\prime})\right\vert <\delta_{3}$; whence
$\left\vert t-t^{\prime}\right\vert <\varepsilon$. Thus $\gamma^{-1}$ is
uniformly continuous on $\gamma\lbrack a,c]$, and we have proved that (i)
implies (ii). A similar argument shows that (i) implies (iii).

Now suppose that (ii) holds, and let $K\subset\lbrack a,b)$ be compact. By
Lemma \ref{nov12l-1}, there exists $c\in\left(  a,b\right)  $ such that
$K\subset\lbrack a,c]$. By (ii), $\gamma^{-1}$ is uniformly continuous on
$\gamma\lbrack a,c]$ and hence on $\gamma(K)$. It follows from Lemma
\ref{dec23p1} that $\gamma(K)$ is compact. Moreover, taking the special case
where $K\subset(a,b)$, we see that (ii) implies \textsf{J2} and hence (i)$\,
$. Likewise, (iii) implies that if $K\subset(a,b]$ is compact, then
$\gamma(K)$ is compact, and (iii) implies (i). It follows that (i), (ii), and
(iii) are equivalent and that the final conclusion of the Proposition holds.%
\hfill

\end{proof}%

\medskip

Moving towards our discussion of the relation between the foregoing definition
of Jordan curve and that of Berg et al. in \cite{JCTthm}, for convenience we
now present some elementary properties of the exponential function. For the
first of these we define the \textbf{positive }$x$\textbf{-axis} as%
\[
X_{+}\equiv\left\{  (x,y)\in\mathbb{C}:x\geq0,\,\ y=0\right\}  .
\]
Note that $\mathbb{C}-X_{+}$ is a simply connected open subset of
$\mathbb{C}-\{0\}$, so there exists a unique differentiable branch $\lambda$
of the logarithmic function such that $\lambda(-1)=i\pi$ \cite[page 143]{BB}.
It follows immediately that for each $z\in\mathsf{U}-\{1\}$ there exists a
unique $\theta\in(0,2\pi)$ such that $z=\exp(i\theta)$: namely, $\theta
=-i\lambda(z)$.

\begin{lemma}
\label{nov18l1}Let $0<\phi<\pi$ and $K=\left\{  \exp(i\theta):\phi\leq
\theta\leq2\pi-\phi\right\}  $. Then

\begin{itemize}
\item[\emph{(i)}] $K$\ is compact and well contained in $\mathbb{C}-X_{+}$;

\item[\emph{(ii)}] for each $\varepsilon>0$ there exists $\delta>0$ such that
if $\theta,\theta^{\prime}\in\left[  \phi,2\pi-\phi\right]  $ and $\left\vert
\exp(i\theta)-\exp(i\theta^{\prime})\right\vert <\delta$, then $\left\vert
\theta-\theta^{\prime}\right\vert <\varepsilon$.
\end{itemize}
\end{lemma}

\begin{proof}
Since $\theta\rightsquigarrow\exp(i\theta)$ is uniformly continuous on the
compact interval $I\equiv\left[  \phi,2\pi-\phi\right]  $, $K$ is totally
bounded. If $\phi\leq\theta\leq\pi/2$ or $3\pi/2\leq\theta\leq2\pi-\phi$,
then
\[
\rho(\exp(i\theta),X_{+})=\left\vert \sin\theta\right\vert \geq\sin\phi>0.
\]
If $\pi/2\leq\theta\leq3\pi/2$, then $\rho(\exp(i\theta),X_{+})=1\geq\sin\phi
$. Since the function $\theta\rightsquigarrow\rho(\exp(i\theta),X_{+})$ is
uniformly continuous on $I$, it follows that $\rho(\exp(i\theta),X_{+}%
)\geq\sin\phi$ for each $\theta\in I$. If $\rho(z,K)\leq\frac{1}{2}\sin\phi$,
then $\rho(z,X_{+})\geq\frac{1}{2}\sin\phi$, so $K\subset\subset
\mathbb{C}-X_{+}$. Let $\lambda$ be the unique differentiable branch of the
logarithmic function on $U$ with $\lambda(-1)=i\pi$. Since $K$ is totally
bounded and well contained in $U$, $\lambda$ is uniformly continuous on $K$,
from which (ii) follows immediately. On the other hand, Proposition
\ref{dec23p1} shows that $\exp(I)$---that is, $K$--- is compact.%
\hfill

\end{proof}

\begin{corollary}
\label{dec20c1}If $0<\theta<2\pi$, then for each $\varepsilon>0$ there exists
$\delta>0$ such that if either $t,t^{\prime}\in\lbrack0,2\pi-\theta]$ or
$t,t^{\prime}\in\lbrack\theta,2\pi],$ and $\left\vert \exp(it)-\exp
(it^{\prime})\right\vert <\delta$, then $\left\vert t-t^{\prime}\right\vert
<\varepsilon$.
\end{corollary}

\begin{proof}
Given $\varepsilon>0$ and taking $\phi=\theta/2$, construct $\delta>0$ as in
Lemma \ref{nov18l1}. If $t,t^{\prime}\in\lbrack0,2\pi-\theta]$ and $\left\vert
\exp(it)-\exp(it^{\prime})\right\vert <\delta$, then $0<\theta/2<\pi$,
$t+\frac{\theta}{2}$ and $t^{\prime}+\frac{\theta}{2}$ are in $\left[
\phi,2\pi-\phi\right]  $, and
\begin{align*}
\left\vert \exp\left(  i\left(  t+\tfrac{\theta}{2}\right)  \right)
-\exp\left(  i\left(  t^{\prime}+\tfrac{\theta}{2}\right)  \right)
\right\vert  &  =\left\vert \exp\left(  \tfrac{i\theta}{2}\right)  \left(
\exp(it)-\exp(it^{\prime})\right)  \right\vert \\
&  =\left\vert \exp(it)-\exp(it^{\prime})\right\vert <\delta,
\end{align*}
so $\left\vert t-t^{\prime}\right\vert <\varepsilon$. On the other hand, if
$t,t^{\prime}\in\lbrack\theta,2\pi]$ and $\left\vert \exp(it)-\exp(it^{\prime
})\right\vert <\delta$, then $t-\frac{\theta}{2}$ and $t^{\prime}-\frac
{\theta}{2}$ belong to $[\phi,2\pi-\phi]$ and%
\begin{align*}
\left\vert \exp\left(  i\left(  t-\tfrac{\theta}{2}\right)  \right)
-\exp\left(  i\left(  t^{\prime}-\tfrac{\theta}{2}\right)  \right)
\right\vert  &  =\left\vert \exp\left(  \tfrac{-i\theta}{2}\right)  \left(
\exp(it)-\exp(it^{\prime})\right)  \right\vert \\
&  =\left\vert \exp(it)-\exp(it^{\prime})\right\vert <\delta,
\end{align*}
so again $\left\vert t-t^{\prime}\right\vert <\varepsilon$.%
\hfill

\end{proof}%

\medskip
It is a simple exercise to prove that for \emph{any} two points $z,z^{\prime}$
of $\mathbb{C}$, if $\exp(z)\neq\exp(z^{\prime})$, then $z\neq z^{\prime}$.

\begin{corollary}
\label{jul30c1}If $t,t^{\prime}\in\lbrack0,2\pi)$ or $t,t^{\prime}\in(0,2\pi
]$, and $t\neq t^{\prime}$, then $\exp(it)\neq\exp(it^{\prime})$.
\end{corollary}

\begin{proof}
For example, if $t,t^{\prime}\in(0,2\pi]$, then setting $\theta=\min\left\{
t,t^{\prime}\right\}  ,$ we have $0<\theta<2\pi$ and $t,t^{\prime}\in
\lbrack\theta,2\pi]$. Taking $\varepsilon=\frac{1}{2}\left\vert t-t^{\prime
}\right\vert $ and constructing $\delta$ as in Corollary \ref{dec20c1}, we see
that $\left\vert t-t^{\prime}\right\vert >\varepsilon$ and therefore
$\left\vert \exp(it)-\exp(it^{\prime})\right\vert \geq\delta$.%
\hfill

\end{proof}%

\medskip

We now define the \textbf{unit circular path}\footnote{%
\normalfont\sf
Bishop calls $\upsilon$ the unit circle.}%
\[
\upsilon:\theta\rightsquigarrow\exp(i\theta)\ \ (0\leq\theta\leq2\pi).
\]
Although every $z\in\mathsf{U}$ can be expressed in the form $z=\exp(i\theta)$
for some $\theta\in\mathbb{R}$, the following Brouwerian example shows that
the classically true equality $\mathsf{U}=\upsilon\lbrack0,2\pi]$ is
essentially nonconstructive.

By Corollary \ref{dec20c1}, there exists $\delta>0$ such that if either
$t,t^{\prime}\in\lbrack0,3\pi/2]$ or $t,t^{\prime}\in\lbrack\pi/2,2\pi]$, and
$\left\vert \exp(it)-\exp(it^{\prime})\right\vert <\delta$, then $\left\vert
t-t^{\prime}\right\vert <\pi/2$. Let $z=(1+ia)/\sqrt{1+a^{2}}$, where $a$ is a
real number such that $\left\vert z-1\right\vert <\delta$; then $z\in
\mathsf{U}$. Suppose that $z=\exp(i\phi)$ for some $\phi\in\left[
0,2\pi\right]  $. If $a\neq0$, then $z\neq1$, so, in view of the remark
preceding Lemma \ref{nov18l1}, there exists a \emph{unique} $\theta\in\left(
0,2\pi\right)  $ such that $z=\exp(i\theta)$. If $\theta\neq\phi$, then
$\phi\notin\left(  0,2\pi\right)  $, so---as either $0<\phi$ or $\phi<2\pi
$---$\,\phi\in\left\{  0,2\pi\right\}  $ and therefore $z=1$, a contradiction.
Thus if $a\neq0$, then $\phi=\theta$ and $0<\phi<2\pi$. In general, either
$\phi>\pi/2$ or $\phi<3\pi/2$. In the first case, if $a>0$, then $\phi,2\pi
\in\lbrack\pi/2,2\pi]$ and $\left\vert \exp(i\phi)-\exp(2\pi i)\right\vert
=\left\vert z-1\right\vert <\delta$, so $\left\vert 2\pi-\phi\right\vert
<\pi/2$; whence $3\pi/2<\phi<2\pi$ and therefore $a=\operatorname{Im}%
z=\sqrt{1+a^{2}}\sin\phi<0$, a contradiction. Hence if $\phi>\pi/2$, then
$\lnot(a>0)$ and therefore $a\leq0$. A similar argument shows that if
$\phi<3\pi/2$, then $\lnot(a<0)$ and therefore $a\geq0$.\ It follows that if
$\mathsf{U}=\mathsf{car}(\upsilon)$, then we can prove that for each
$x\in\mathbb{R}$ either $x\geq0$ or $x\leq0$, a proposition which is
equivalent to the essentially nonconstructive omniscience principle\footnote{%
\normalfont\sf
See \cite[pages 9-10]{BVtech}.}

\begin{quote}
\textbf{LLPO\ \ }For each binary sequence $\left(  a_{n}\right)  _{n\geq1}$
with at most one term equal to $1$, either $a_{n}=0$ for all even $n$ or else
$a_{n}=0$ for all odd $n$.
\end{quote}%

\noindent
The best we can hope for in this context is therefore:

\begin{lemma}
\label{nov18l2}The range of $\upsilon$ is dense in $\mathsf{U}$.
\end{lemma}

\begin{proof}
Either $z\neq1$ and $z=\exp(i\theta)\ $for a unique $\theta\in\left(
0,2\pi\right)  $, or else $\left\vert z-\exp(i0)\right\vert =\left\vert
z-1\right\vert <\varepsilon$.%
\hfill

\end{proof}

\begin{proposition}
\label{dec21p1}The unit circular path $\upsilon$ is a Jordan curve with
carrier $\mathsf{U}$.
\end{proposition}

\begin{proof}
Certainly, $\upsilon$ is a closed path with parameter interval $\left[
0,2\pi\right]  $. By Corollary \ref{jul30c1} and Proposition \ref{jun12p1},
$\upsilon\ $satisfies \textsf{J1}. By Corollary \ref{dec20c1} and Proposition
\ref{nov22p2}, $\upsilon$ satisfies \textsf{J2}. Thus $\upsilon$ is a Jordan
curve. It follows from Lemma \ref{nov18l2} that $\mathsf{car}(\upsilon
)=\mathsf{U}$.%
\hfill

\end{proof}

\begin{lemma}
\label{nov25l1}Let $X,Y$ be metric spaces, and $D$ a dense subspace of $X$.
Let $F$ be a uniformly continuous mapping of $X$ into $Y$ such that

\begin{enumerate}
\item[\emph{(i)}] the restriction $f$ of $F$ to $D\ $is injective and

\item[\emph{(ii)}] $f^{-1}:f(D)\rightarrow X$ is uniformly continuous.
\end{enumerate}%

\noindent
Then $F$ is injective and $F^{-1}:F(X)\rightarrow X$ is uniformly continuous.
If also $X$ is complete, then $F(X)$ is the closure of $f(D)$ in $Y$.
\end{lemma}

\begin{proof}
Let $x,x^{\prime}\in X$ and choose sequences $\left(  x_{n}\right)  _{n\geq
1},(x_{n}^{\prime})_{n\geq1}$ in $D$ converging to $x$ and $x^{\prime}$
respectively. First consider the case where $x\neq x^{\prime}$. By (ii), there
exists $\delta>0$ such that if $z,z^{\prime}\in D$ and $\rho(f(z),f(z^{\prime
}))<\delta$, then $\rho(z,z^{\prime})<\frac{1}{2}\rho(x,x^{\prime})$. Choose
$N$ such that
\[
\max\{\rho(x_{n},x),\rho(x_{n}^{\prime}{},x^{\prime})\}<\tfrac{1}{4}%
\rho(x,x^{\prime})\ \ \ \ \left(  n\geq N\right)  .
\]
Then for such $n,$%
\[
\rho(x_{n},x_{n}^{\prime})\geq\rho(x,x^{\prime})-\rho(x_{n},x)-\rho
(x_{n}^{\prime}{},x^{\prime})>\tfrac{1}{2}\rho(x,x^{\prime}),
\]
so%
\[
\rho(F(x),F(x^{\prime}))=\lim_{n\rightarrow\infty}\rho(f(x_{n}),f(x_{n}%
^{\prime}))\geq\delta
\]
and therefore $F(x)\neq F(x^{\prime})$. Hence $F$ is injective, $F^{-1}%
:F(x)\rightsquigarrow x$ is a function from $F(X)$ onto $X$, and the
restriction of $F^{-1}$ to $f(D)$ equals $f$.

Next, let $\varepsilon>0$ and choose $\delta_{1}>0$ such that if $z,z^{\prime
}\in D$ and $\rho(f(z),f(z^{\prime}))<\delta_{1}$, then $\rho(z,z^{\prime
})<\varepsilon$. This time, consider the case where $\rho(F(x),F(x^{\prime
}))<\delta_{1}$. Since $F$ is uniformly continuous, for all sufficiently large
$N$ we have $\rho(F(x_{n}),F(x_{n}^{\prime}))=\rho(f(x_{n}),f(x_{n}^{\prime
}))<\delta_{1}$ and therefore $\rho(x_{n},x_{n}^{\prime})<\varepsilon$; so%
\[
\rho(x,x^{\prime})=\lim_{n\rightarrow\infty}\rho(x_{n},x_{n}^{\prime}%
)\leq\varepsilon.
\]
Since $\varepsilon>0$ is arbitrary, it follows that $F^{-1}$ is uniformly
continuous on $F(X)$. Now suppose that $X$ is complete, and let $y\in
\overline{F(X)}\subset Y$. Then there is a sequence $\left(  y_{n}\right)
_{n\geq1}$ in $F(X)\ $that converges to $y$. Since $F^{-1}$ is uniformly
continuous, $\left(  F^{-1}(y_{n})\right)  _{n\geq1}$ is a Cauchy sequence in
the complete space $X$ and so converges to a limit $x_{\infty}\in X$. Then
$F(x_{\infty})=\lim_{n\rightarrow\infty}F(F^{-1}(y_{n}))=y$.%
\hfill

\end{proof}

\begin{proposition}
\label{dec21p2}Let $\gamma:[a,b]\rightarrow\mathbb{C}$ be a Jordan curve, and
$F$ a continuous injective mapping of $\mathsf{car}(\gamma)$ into $\mathbb{C}$
such that $F^{-1}$ is uniformly continuous on $F(\mathsf{car}(\gamma))$. Then
$F\circ\gamma:\left[  a,b\right]  \rightarrow\mathbb{C}$ is a Jordan curve
with carrier $F(\mathsf{car}(\gamma))$.
\end{proposition}

\begin{proof}
It is straightforward to show that $F\circ\gamma$ is uniformly continuous,
that $F\circ\gamma(a)=F\circ\gamma(b)$, and that $F\circ\gamma$ satisfies
\textsf{J1}. Consider a compact subset $K$ of $F\circ\gamma(a,b)$. The
restriction of $F^{-1}$ to $K$ is continuous and injective, $F^{-1}%
(K)\subset\gamma(a,b)$, and the restriction of $F$ to $F^{-1}(K)$ is uniformly
continuous. Hence, by Proposition \ref{dec23p1}, $F^{-1}(K)$ is a compact
subset of $\gamma\left(  a,b\right)  $, so $\gamma^{-1}$ is uniformly
continuous on $F^{-1}(K)$, by \textsf{J2}, and therefore $\left(  F\circ
\gamma\right)  ^{-1}=\gamma^{-1}\circ F^{-1}$ is uniformly continuous on $K$.
Thus $F\circ\gamma$ satisfies \textsf{J2} and is therefore a Jordan curve.
Now, by continuity, $F\circ\gamma(a,b)$ is dense in both $F\circ
\mathsf{car}(\gamma)$ and $\mathsf{car}(F\circ\gamma)$. Since the last-named
set is compact, by Lemma \ref{dec23p1}, and therefore closed, it follows that
$F\circ\mathsf{car}(\gamma)=\mathsf{car}(F\circ\gamma)$.%

\hfill

\end{proof}%

\medskip

From this and Proposition \ref{dec21p1}, we obtain the first of two results
that, taken together, will link our and the Berg et. al notions of Jordan curve.

\begin{corollary}
\label{nov26p1}Let $F:\mathsf{U}\rightarrow\mathbb{C}$ be uniformly
continuous, injective, and such that $F^{-1}$ is uniformly continuous. Then
$F\circ\upsilon:t\rightsquigarrow F(\exp(it))\ \ (0\leq t\leq2\pi)$ is a
Jordan curve such that $\mathsf{car}(F\circ\upsilon)=F(\mathsf{U})$.
\end{corollary}%

\medskip
\noindent
The second linking result is this.

\begin{theorem}
\label{nov18t1}Let $\gamma$ be a Jordan curve with parameter interval $\left[
0,2\pi\right]  $. Then

\begin{enumerate}
\item[\emph{(i)}] $f:\exp(it)\rightsquigarrow\gamma(t)$ $(0\leq t\leq2\pi
)\ $is a uniformly continuous injective mapping of the range of $\upsilon$
onto $\gamma\lbrack0,2\pi]$ with uniformly continuous inverse, and

\item[\emph{(ii)}] $f$ extends to a continuous injective mapping $F$ of
$\mathsf{U}$ onto $\mathsf{car}(\gamma)$ with uniformly continuous inverse.
\end{enumerate}
\end{theorem}

\begin{proof}
We first show that $f$ is a well-defined function. If $t,t^{\prime}\in\left[
0,2\pi\right]  $ and $\exp(it)=\exp(it^{\prime})$, then there exists an
integer $n$ such that $t^{\prime}=t+2n\pi$ \cite[page 143]{BB}. In this case,
$n\in\left\{  -1,0,1\right\}  $, so we have three alternatives---$\,t=2\pi$
and $t^{\prime}=0\,$; $t=t^{\prime}$; $t^{\prime}=2\pi$ and $t=0\,$---in each
of which $\gamma(t)=\gamma(t^{\prime})$. Thus $f$ is indeed a function. On the
other hand, if $t,t^{\prime}\in\left[  0,2\pi\right]  $ and $\exp(it)\neq
\exp(it^{\prime})$, then $t\neq t^{\prime}$ and either $\exp(it)\neq1$ or
$\exp(it^{\prime})\neq1$. Taking, to illustrate, the latter alternative, we
have $0<t^{\prime}<2\pi$, so $(t,t^{\prime})\in\left[  0,2\pi\right]
\times(0,2\pi)$ and therefore, by \textsf{J1}\textbf{, }$t\neq t^{\prime}$;
whence $f(\exp(it))=\gamma(t)\neq\gamma(t^{\prime})=f(\exp(it^{\prime}))$.
This shows that $f$ is injective. Clearly, $f$ maps the range of $\upsilon$
onto $\gamma\left[  0,2\pi\right]  $.

To prove that $f$ is uniformly continuous on the range of $\upsilon$, given
$\varepsilon>0$ choose $\delta\in\left(  0,\pi/2\right)  $ such that if
$t,t^{\prime}\in\left[  0,2\pi\right]  $ and $\left\vert t-t^{\prime
}\right\vert <4\delta$, then $\left\vert \gamma(t)-\gamma(t^{\prime
})\right\vert <\varepsilon$. By Lemma \ref{nov18l1}, there exists $\delta
_{1}\in\left(  0,\delta\right)  $ such that if $\delta\leq t,t^{\prime}%
\leq2\pi-\delta$ and $\left\vert \exp(it)-\exp(it^{\prime})\right\vert
<\delta_{1}$, then $\left\vert t-t^{\prime}\right\vert <\delta$ and therefore
$\left\vert \gamma(t)-\gamma(t^{\prime})\right\vert <\varepsilon$. Let
$t,t^{\prime}\in\left[  0,2\pi\right]  $ and $\left\vert \exp(it)-\exp
(it^{\prime})\right\vert <\delta_{1}$. Either $t<3\delta$ or else $t>2\delta$.
In the first case, since $\left\vert t-t^{\prime}\right\vert <\delta$, we have
$t,t^{\prime}<4\delta$, so%
\[
\left\vert \gamma(t)-\gamma(t^{\prime})\right\vert \leq\left\vert
\gamma(t)-\gamma(0)\right\vert +\left\vert \gamma(t^{\prime})-\gamma
(0)\right\vert <2\varepsilon.
\]
In the case $t>2\delta$, we have $t^{\prime}>\delta$. Then either
$t>2\pi-3\delta$ or $t<2\pi-2\delta$. In the first case, $t,t^{\prime}\in
(2\pi-4\delta,2\pi)$, so%
\[
\left\vert \gamma(t)-\gamma(t^{\prime})\right\vert \leq\left\vert
\gamma(t)-\gamma(2\pi)\right\vert +\left\vert \gamma(t^{\prime})-\gamma
(2\pi)\right\vert <2\varepsilon.
\]
In the second case, $t^{\prime}<2\pi-\delta$, so $t,t^{\prime}\in\left(
\delta,2\pi-\delta\right)  $ and $\left\vert \exp(it)-\exp(it^{\prime
})\right\vert <\delta_{1}$, and therefore $\left\vert \gamma(t)-\gamma
(t^{\prime})\right\vert <\varepsilon$. We now see that%
\[
\left\vert f(\exp(it))-f(\exp(it^{\prime}))\right\vert =\left\vert
\gamma(t)-\gamma(t^{\prime})\right\vert <2\varepsilon
\]
for all $t,t^{\prime}\in\left[  0,2\pi\right]  $ with $\left\vert
\exp(it)-\exp(it^{\prime})\right\vert <\delta_{1}$. Since $\varepsilon>0$ is
arbitrary, this proves that $f$ is uniformly continuous on the range of
$\upsilon$.

Next we prove that $f^{-1}:\gamma(t)\rightsquigarrow\exp(it)$ is uniformly
continuous on $\gamma\lbrack0,2\pi]$. Again letting $\varepsilon>0$, since the
function $t\rightsquigarrow\exp(it)$ is continuous on $\left[  0,2\pi\right]
$ we can find $\delta_{2}\in(0,\pi/4)\ $such that if $t,t^{\prime}\in\left[
0,2\pi\right]  $ and $\left\vert t-t^{\prime}\right\vert <4\delta_{2}$, then
$\left\vert \exp(it)-\exp(it^{\prime})\right\vert <\varepsilon$. Let $\ J$ be
the compact interval $\left[  \delta_{2},2\pi-\delta_{2}\right]  $. By
Proposition \ref{nov22p2}, there exists $\delta_{3}>0$ such that $\left\vert
t-t^{\prime}\right\vert <\delta_{2}$ whenever $t,t^{\prime}\in J$ and
$\left\vert \gamma(t)-\gamma(t^{\prime})\right\vert <\delta_{3}$. Consider
$t,t^{\prime}\in\lbrack0,2\pi]$ with $\left\vert \gamma(t)-\gamma(t^{\prime
})\right\vert <\delta_{3}$. We have these three alternatives: $t\in
\lbrack0,3\delta_{2})\ $or $t\in\lbrack2\delta_{2},2\pi-2\delta_{2}]$ or
$t\in(2\pi-3\delta_{2},2\pi]$. If $t\in\lbrack0,3\delta_{2})$, then
$t,t^{\prime}\in\lbrack0,4\delta_{2})$, so%
\[
\left\vert \exp(it)-\exp(it^{\prime})\right\vert \leq\left\vert \exp
(it)-1\right\vert +\left\vert \exp(it^{\prime})-1\right\vert <2\varepsilon;
\]
likewise, if $t\in(2\pi-3\delta_{2},2\pi]$, then $t,t^{\prime}\in(2\pi
-4\delta_{2},2\pi]$ and $\left\vert \exp(it)-\exp(it^{\prime})\right\vert
<2\varepsilon$. If $t\in\lbrack2\delta_{2},2\pi-2\delta_{2}]$, then
$t,t^{\prime}\in J$ so, since $\left\vert \gamma(t)-\gamma(t^{\prime
})\right\vert <\delta_{3}$, we have $\left\vert t-t^{\prime}\right\vert
<\delta_{2}$ and therefore $\left\vert \exp(it)-\exp(it^{\prime})\right\vert
<\varepsilon$. Hence, $\varepsilon>0$ being arbitrary, $f^{-1}$ is uniformly
continuous on $\gamma\lbrack0,2\pi]$. This completes the proof of (i). It
follows from (i), Lemma \ref{nov18l2}, the completeness of $\mathsf{car}%
(\gamma)$, and Lemma \ref{nov25l1} that $F$ extends to a uniformly continuous
injection $F$ of $\mathsf{U}$ into $\mathsf{car}(\gamma)$ and that (since
$\mathsf{U}$ is complete) $F(\mathsf{U})=\mathsf{car}(\gamma)$. This proves
(ii).%
\hfill

\end{proof}%

\medskip

Theorem \ref{nov18t1} tells us that each Jordan curve $\gamma\ $on $\left[
0,2\pi\right]  \ $gives rise to a unique homeomorphism $\Phi(\gamma)$ of
$\mathsf{U}$ onto $\mathsf{car}(\gamma)$ such that $\Phi(\gamma)\circ
\upsilon=\gamma$. On the other hand, by Corollary \ref{nov26p1}, for each
homeomorphism $F$ of $\mathsf{U}$ onto a (perforce compact) subset of
$\mathbb{C}$, $\Psi(F)\equiv F\circ\upsilon$ defines a Jordan curve with
parameter interval $\left[  0,2\pi\right]  $ and carrier $F(\mathsf{U})$. If
$\gamma$ is a Jordan curve, then%
\[
\Psi(\Phi(\gamma))=\Phi(\gamma)\circ\upsilon=\gamma.
\]
On the other hand, if $F$ is a homeomorphism of $\mathsf{U}$ onto
$F(\mathsf{U})\subset\mathbb{C}$, then
\[
\Phi(\Psi(F))\circ\upsilon=\Psi(F)=F\circ\upsilon,
\]
so the uniformly continuous functions $\Phi\circ\Psi(F)$ and $F$ on
$\mathsf{U}$ are equal on the range of $\upsilon$; since, by Lemma
\ref{nov18l2}, that range is dense in $\mathsf{U}$, we conclude that
$\Phi\circ\Psi(F)=F$. Thus $\Phi$ is a one-one function from the set
$\mathfrak{J~}$of Jordan curves with parameter interval $\left[
0,2\pi\right]  $ onto the set $\mathfrak{H~}$of complex-valued homeomorphisms
with domain $\mathsf{U}$, and $\Psi$ is the inverse function from
$\mathfrak{H}$ onto $\mathfrak{J}$; also, by Theorem \ref{nov18t1}, for each
$\gamma\in\mathfrak{J}$ the range of $\Phi(\gamma)$ is $\mathsf{car}(\gamma)$,
and, by Corollary \ref{nov26p1}, for each $F\in\mathfrak{H}$ the carrier of
$\Psi(F)$ is the range of $F$. This provides a one-one correspondence between
our notion of Jordan curve and the Berg one, such that the carrier of a Jordan
curve equals the range of the corresponding homeomorphism of $\mathsf{U}$.
Using this correspondence and the equivalence of any given Jordan curve with
one that has parameter interval $\left[  0,2\pi\right]  $, we can---and
shall---apply the Jordan Curve Theorem and its consequences in \cite{JCTthm}
to our notion of Jordan curve.

We note here two things:

\begin{itemize}
\item[a)] The index\footnote{%
\normalfont\sf
Since it is the \emph{properties} of the index that matter to us, we give only
the relevant ones of those and refrain from distracting reader with the
detailed construction of the index given in \cite{JCTthm}.} (as defined in
\cite{JCTthm}) of a point with respect to a Berg Jordan curve depends only on
the range of the curve: if $F_{1}$ and $F_{2}$ are two Berg Jordan curves with
the same range $S$, then the index of $\zeta\in\mathbb{C}-S$ with respect to
$F_{1}$ equals the index of $\zeta$ with respect to $F_{2}$.

\item[b)] When we refer to the \textbf{index} of a point $\zeta\in
-\mathsf{car}(\gamma)$ with respect to a Jordan curve $\gamma$ (in our sense),
we mean the index of $\zeta$ with respect to the Berg Jordan curve $\left(
\mathsf{car}(\gamma),\Phi(\widehat{\gamma})\right)  $, where $\widehat{\gamma
}$ is the unique Jordan curve (in our sense) on $\left[  0,2\pi\right]  $ that
is equivalent to $\gamma$; we denote this index by $\mathsf{ind}(\zeta
;\gamma)$, in contrast to the notation $\mathsf{ind}(\zeta;\Phi
(\widehat{\gamma}))$ that Berg et al. would have used.
\end{itemize}

With all this in mind, we see from JCT that a Jordan curve (in our sense)
$\gamma$ divides $-\mathsf{car}(\gamma)$ into two path connected components:%
\begin{align*}
\mathsf{in}(\gamma) &  \equiv\left\{  z\in-\mathsf{car}(\gamma):\mathsf{ind}%
(z;\gamma)=1\right\} \\
\mathsf{out}(\gamma) &  \equiv\left\{  z\in-\mathsf{car}(\gamma):\mathsf{ind}%
(z;\gamma)=0\right\}  .
\end{align*}
Moreover, by JCT and the corollary to Proposition 4 of \cite{JCTthm}, two
points $z,z^{\prime}$ of $-\mathsf{car}(\gamma)$ have the same index if and
only if they can be joined by a polygonal path whose carrier is bounded away
from $\gamma$. A point $z\in-\mathsf{car}(\gamma)$ is said to be
\textbf{inside~}(respectively, \textbf{outside})~$\gamma$ if $z\in
\mathsf{in}(\gamma)\ $(respectively, $z\in\mathsf{out}(\gamma)$). The proof of
JCT in \cite{JCTthm} shows that all points sufficiently far away from
$\mathsf{car}(\gamma)$ have index $0$. Theorem 2 of \cite{JCTthm} shows that
for each $z\in\mathsf{car}(\gamma)$ and each $\varepsilon>0$, there exist
points $a,b$ within $\varepsilon~$of $z$ such that $\mathsf{ind}(a;\gamma)=1$
and $\mathsf{ind}(b;\gamma)=0$. It follows that $\mathsf{in}(\gamma
)\cup\mathsf{car}(\gamma)\cup\mathsf{out}(\gamma)$ is dense in $\mathbb{C}$,
that $\overline{\mathsf{in}(\gamma)}=\overline{\mathsf{in}(\gamma
)\cup\mathsf{car}(\gamma)}$, and that $\overline{\mathsf{out}(\gamma
)}=\overline{\mathsf{out}(\gamma)\cup\mathsf{car}(\gamma)}$. We call
$\overline{\mathsf{in}(\gamma)}$ the \textbf{region bounded by the Jordan
curve} $\gamma$ and say that $\gamma$ \textbf{bounds} it.

\begin{proposition}
\label{oct15p1}If $\gamma~$is a Jordan curve, $\zeta\in-\mathsf{car}(\gamma)$,
and $r=\rho(\zeta,\mathsf{car}(\gamma))$, then\footnote{\textsf{We denote the
open and closed balls in }$C\ $\textsf{with centre }$a$\textsf{\ and radius
}$r$\textsf{\ in a metric space by }$B(a,r)$\textsf{\ and }$\overline{B}%
(a,r)$\textsf{\ respectively.}} $B(\zeta,r)\subset-\mathsf{car}(\gamma)$ and
$\mathsf{ind}(z;\gamma)=\mathsf{ind}(\zeta;\gamma)$ for each $z\in B(\zeta,r)$.
\end{proposition}

\begin{proof}
Given $z\in B(\zeta,r)$, let $s=r-\left\vert z-\zeta\right\vert >0$. Then for
each $\lambda\in\left[  0,1\right]  $ and each $t\in\left[  a,b\right]  $,%
\[
\left\vert (1-\lambda)\zeta+\lambda z-\gamma(t)\right\vert \geq\left\vert
\zeta-\gamma(t)\right\vert -\lambda\left\vert z-\zeta\right\vert \geq
r-\left\vert z-\zeta\right\vert =s.
\]
Thus the polygonal path $\mathsf{lin}(\zeta,z)$ joining $\zeta$ and $z$ is
bounded away from $\mathsf{car}(\gamma)$, so $z\in-\mathsf{car}(\gamma)$ and
$\mathsf{ind}(z;\gamma)=\mathsf{ind}(\zeta;\gamma)$.%
\hfill

\end{proof}

\begin{corollary}
\label{dec02c1}If $\gamma$ is a Jordan curve, then $\mathsf{in}(\gamma)$ and
$\mathsf{out}(\gamma)$ are open in $\mathbb{C}$.
\end{corollary}

\begin{corollary}
\label{dec02c2}If $\gamma$ is a Jordan curve, then%
\[
\overline{\mathsf{in}(\gamma)}\cap\mathsf{out}(\gamma)=\varnothing
=\mathsf{in}(\gamma)\cap\overline{\mathsf{out}(\gamma)}%
\]
and $\overline{\mathsf{in}(\gamma)}\cap\overline{\mathsf{out}(\gamma
)}=\mathsf{car}(\gamma)$.
\end{corollary}

\begin{proof}
For the first part of the conclusion, consider, for example, $\zeta
\in\mathsf{out}(\gamma)$, and let $r=\rho(\zeta,\mathsf{car}(\gamma))$. If
also $\zeta\in\overline{\mathsf{in}(\gamma)}$, then there exists $z\in
B(\zeta,r)\cap\mathsf{in}(\gamma)$; but then, by Proposition \ref{oct15p1},
$\mathsf{ind}(z;\gamma)=\mathsf{ind}(\zeta;\gamma)=0$, which is absurd. For
the remaining part of the conclusion, first note that\ if $z\in\overline
{\mathsf{in}(\gamma)}\cap\overline{\mathsf{out}(\gamma)}$ and $\rho
(z,\mathsf{car}(\gamma))>0$, then $z\in\overline{\mathsf{in}(\gamma)}%
\cap\mathsf{out}(\gamma)$ or else $z\in\mathsf{in}(\gamma)\cap\overline
{\mathsf{out}(\gamma)}$, a contradiction in either case. Hence $\rho
(z,\mathsf{car}(\gamma))=0$ and $z$ is in the closed set $\mathsf{car}(\gamma)
$. On the other hand, if $z\in\mathsf{car}(\gamma)$, then by \textsf{JCT},
there are points of $\mathsf{in}(\gamma)$ and points of $\mathsf{out}(\gamma)$
arbitrarily close to $z$, so $z\in\overline{\mathsf{in}(\gamma)}\cap
\overline{\mathsf{out}(\gamma)}$.%
\hfill

\end{proof}

\begin{proposition}
\label{mar7p1}Let $\gamma:\left[  a,b\right]  \rightarrow\mathbb{C}$ be a
Jordan curve, let $\xi,\eta$ be distinct complex numbers such that $\xi
\in\overline{\mathsf{in}(\gamma)}$ and $\eta\in\overline{\mathsf{out}(\gamma
)}$, and let $\mu~$be a path joining $\xi$ and $\eta$. Then for each
$\varepsilon>0$ there exists $t\in\left[  a,b\right]  $ such that $\rho
(\gamma(t),\mathsf{car}(\mu))<\varepsilon$.
\end{proposition}

\begin{proof}
Since $\mathsf{car}(\gamma)$ and $\mathsf{car}(\mu)$ are compact and the
function $z\rightsquigarrow\rho(z,\mathsf{car}(\gamma))$ is continuous on
$\mathsf{car}(\mu)$,%
\[
m\equiv\inf\left\{  \rho(z,\mathsf{car}(\gamma)):z\in\mathsf{car}%
(\mu)\right\}
\]
exists. Given $\varepsilon>0$, we have either $m>\varepsilon/2$ or
$m<\varepsilon$. In the former case, $\mu$ is bounded away from $\mathsf{car}%
(\gamma)$, so by the Corollary to Proposition 4 of \cite{JCTthm},
$\mathsf{ind}(\xi;\gamma)=\mathsf{ind}(\eta;\gamma)$. Hence either $\xi
,\eta\in\mathsf{in}(\gamma)$ or else $\xi,\eta\in\mathsf{out}(\gamma)$. In
view of Corollary \ref{dec02c2}, the first case is ruled out since $\eta
\in\overline{\mathsf{out}(\gamma)}$, and the second is ruled out since $\xi
\in\overline{\mathsf{in}(\gamma)}$. Hence $m\not > \varepsilon/2$, so
$m<\varepsilon$ and therefore the desired $t\in\left[  a,b\right]  $ exists.%
\hfill

\end{proof}%

\medskip

Recall that in BISH a metric space is called\textbf{\ locally compact }if each
bounded subset of\textbf{\ }$X$ is contained in a compact set. Both
$\mathbb{R}$ and $\mathbb{C}$ are locally compact.

\begin{proposition}
\label{mar6p2}If $\gamma$ is a Jordan curve in $\mathbb{C}$, then

\begin{enumerate}
\item[\emph{(i)}] $\overline{\mathsf{in}(\gamma)}$ is compact and
$\overline{\mathsf{out}(\gamma)}$ is locally compact;

\item[\emph{(ii)}] $\rho(z,\mathsf{in}(\gamma))=\rho(z,\mathsf{car}(\gamma))$
for all $z\in\overline{\mathsf{out}(\gamma)}$;

\item[\emph{(iii)}] $\rho(z,\mathsf{out}(\gamma))=\rho(z,\mathsf{car}%
(\gamma))$ for all $z\in\overline{\mathsf{in}(\gamma)}$.
\end{enumerate}
\end{proposition}

\begin{proof}
Consider $z\in\overline{\mathsf{out}(\gamma)}$. By Proposition \ref{mar7p1},
for each $z^{\prime}\in\mathsf{in}(\gamma)$ the path $\mathsf{lin}%
(z,z^{\prime})$ comes arbitrarily close to $\mathsf{car}(\gamma)$, so there
exists $\xi\in\lbrack z,z^{\prime}]$ such that $\rho(\xi,\mathsf{car}%
(\gamma))$ is arbitrarily small; since%
\[
\rho(z,\mathsf{car}(\gamma))\leq\left\vert z-\xi\right\vert +\rho
(\xi,\mathsf{car}(\gamma))\leq\left\vert z-z^{\prime}\right\vert +\rho
(\xi,\mathsf{car}(\gamma)),
\]
it follows that $\rho(z,\mathsf{car}(\gamma))\leq\left\vert z-z^{\prime
}\right\vert $. On the other hand, for each $\varepsilon>0$ there exists
$\zeta\in\mathsf{car}(\gamma)$ such that $\left\vert z-\zeta\right\vert
<\rho(z,\mathsf{car}(\gamma))+\varepsilon/2$; in turn, by Theorem 2 of
\cite{JCTthm}, there exists $\zeta^{\prime}\in\mathsf{in}(\gamma)$ such that
$\left\vert \zeta-\zeta^{\prime}\right\vert <\varepsilon/2$ and therefore
$\left\vert z-\zeta^{\prime}\right\vert <\rho(z,\mathsf{car}(\gamma
))+\varepsilon$. It follows that $\rho(z,\mathsf{in}(\gamma))$ exists and
equals $\rho(z,\mathsf{car}(\gamma))$. This proves (ii). Moreover, it easily
follows that $\rho(z,\mathsf{in}(\gamma))$ exists for all $z$ in the dense
subset $\mathsf{car}(\gamma)\cup\mathsf{in}(\gamma)\cup\mathsf{out}(\gamma)$
of $\mathbb{C}$; whence $\mathsf{in}(\gamma)$, and therefore $\overline
{\mathsf{in}(\gamma)}$, is located in $\mathbb{C}$. To prove that
$\mathsf{in}(\gamma)$ is bounded, fix $a\in\mathsf{in}(\gamma)$ and choose
$R>0$ such that the compact set $\mathsf{car}(\gamma)$ is well contained in
$\overline{B}(a,R)$. As noted before Proposition \ref{oct15p1}, there exists
$\zeta^{\prime\prime}\in\mathsf{out}(\gamma)$ such that $\left\vert
\zeta^{\prime\prime}-a\right\vert >R$. If $z\in\mathsf{in}(\gamma)$ and
$\left\vert z-a\right\vert >R$, then we can join $z $ and $\zeta^{\prime
\prime}$ by a path that is bounded away from $\overline{B}(a,R)$; by
Proposition \ref{mar7p1}, this path must come arbitrarily close to
$\mathsf{car}(\gamma)$, which is impossible since $\mathsf{car}(\gamma
)\subset\subset\overline{B}(a,R)$. Hence $\mathsf{in}(\gamma)$ is a subset of
$\overline{B}(a,R)$. It follows that $\overline{\mathsf{in}(\gamma)}$ is a
closed, located subset of the compact set $\overline{B}(a,R)$ and is therefore
compact \cite[p. 89, Proposition 7]{Bishop}. Finally, arguments like those
used above enable us to prove that (iii) holds, that $\overline{\mathsf{out}%
(\gamma)}$ is closed and located in $\mathbb{C}$, and hence, using
\cite[(4.6.3)]{BB}, that $\overline{\mathsf{out}(\gamma)}$ is locally compact.
This completes the proof of (i).%
\hfill

\end{proof}

\begin{corollary}
\label{aug02c1}If $\gamma$ is a Jordan curve in $\mathbb{C}$, then
$\mathsf{in}(\gamma)=-\overline{\mathsf{out}(\gamma)}$ and $\mathsf{out}%
(\gamma)=-\overline{\mathsf{in}(\gamma)}$.
\end{corollary}

\begin{proof}
By Proposition \ref{mar6p2}, for all $z\in\mathsf{in}(\gamma)$,
\[
\rho(z,\overline{\mathsf{out}(\gamma)})=\rho(z,\mathsf{out}(\gamma
))=\rho(z,\mathsf{car}(\gamma))>0;
\]
from which it follows that $\mathsf{in}(\gamma)\subset$~ $-\overline
{\mathsf{out}(\gamma)}$. On the other hand, if $z\in-\overline{\mathsf{out}%
(\gamma)}$, then $\rho(z,\mathsf{car}(\gamma))>0$, so $z\ $belongs to
$(\mathsf{in}(\gamma)\cup\mathsf{out}(\gamma))\cap$ $-\overline{\mathsf{out}%
(\gamma)}\ $and therefore, by Corollary \ref{dec02c2}, to $\mathsf{in}%
(\gamma)$; whence $-\overline{\mathsf{out}(\gamma)}\subset\mathsf{in}(\gamma)$
and therefore $\mathsf{in}(\gamma)=$~ $-\overline{\mathsf{out}(\gamma)}$. The
proof that $\mathsf{out}(\gamma)=-\overline{\mathsf{in}(\gamma)}$ is similar.%
\hfill

\end{proof}

\begin{proposition}
\label{oct29p1}Let $\gamma$ be a Jordan curve in $\mathbb{C}$. If $\zeta
\in\overline{\mathsf{in}(\gamma)}$ \emph{(}respectively, $\zeta\in
\overline{\mathsf{out}(\gamma)}$\emph{)}, then $\overline{B}(\zeta,\rho
(\zeta,\mathsf{car}(\gamma)))$ is a subset of $\overline{\mathsf{in}(\gamma)}$
(respectively, $\overline{\mathsf{out}(\gamma)}$).
\end{proposition}

\begin{proof}
By Proposition \ref{mar6p2}(i), $\overline{\mathsf{in}(\gamma)}$ is compact
and hence located. Let $\zeta\in\overline{\mathsf{in}(\gamma)}$ and
$r=\rho(\zeta,\mathsf{car}(\gamma))$. Given $z\in\overline{B}(\zeta,r)$,
suppose that $\rho(z,\overline{\mathsf{in}(\gamma)})>0$. Then $z\neq\zeta$, so
$r>0$, $\zeta\in-\mathsf{car}(\gamma)$, and therefore, by Proposition
\ref{oct15p1}, $B(\zeta,r)\subset\mathsf{in}(\gamma)$; whence $z\in
\overline{\mathsf{in}(\gamma)}$, a contradiction. Hence $\rho(z,\overline
{\mathsf{in}(\gamma)})=0$ and therefore $z\in\overline{\mathsf{in}(\gamma)}$.
Thus $\overline{B}(\zeta,\rho(\zeta,\mathsf{car}(\gamma)))\subset
\overline{\mathsf{in}(\gamma)}$. A similar argument proves that if $\zeta
\in\overline{\mathsf{out}(\gamma)}$, then $\overline{B}(\zeta,\rho
(\zeta,\mathsf{car}(\gamma)))\subset\overline{\mathsf{out}(\gamma)} $.%
\hfill

\end{proof}

\begin{corollary}
\label{mar6l3}Let $\gamma$ be a Jordan curve, $\zeta\in\mathsf{in}(\gamma)$,
and $r=\rho(\zeta,\mathsf{car}(\gamma))$. Then for each $\varepsilon>0$ there
exists $z\in\overline{\mathsf{in}(\gamma)}$ such that $\left\vert
z-\zeta\right\vert =r$ and $\rho(z,\mathsf{car}(\gamma))<\varepsilon$.
\end{corollary}

\begin{proof}
Given $\varepsilon>0$, find $t\ $in the parameter interval of $\gamma$ such
that $0<r\leq\left\vert \zeta-\gamma(t)\right\vert <r+\varepsilon$. Let
$\alpha=r/\left\vert \zeta-\gamma(t)\right\vert $ and $z=(1-\alpha
)\zeta+\alpha\gamma(t)$. Then $\left\vert z-\zeta\right\vert =\alpha\left\vert
\zeta-\gamma(t)\right\vert =r$, so $z\in\overline{B}(\zeta,r)$ and therefore,
by Proposition \ref{oct29p1}, $z\in\overline{\mathsf{in}(\gamma)}$. On the
other hand,%
\[
\left\vert z-\gamma(t)\right\vert =\left(  1-\alpha\right)  \left\vert
\zeta-\gamma(t)\right\vert =\left\vert \zeta-\gamma(t)\right\vert
-r<\varepsilon,
\]
so $\rho(z,\mathsf{car}(\gamma))<\varepsilon$.%
\hfill

\end{proof}%

\medskip

If $K$ is a compact subset of $\mathbb{C}$, and $f:K\rightarrow\mathbb{C}$ is
continuous, we define%
\[
m(f,K)\equiv\inf\{\left\vert f(z)\right\vert :z\in K\}\text{ \ and
\ }\left\Vert f\right\Vert _{K}\equiv\sup\{\left\vert f(z)\right\vert :z\in
K\},
\]
which exist since $f(K)$ is totally bounded. If also $B$ is a totally bounded
subset of $K$ such that $\overline{B}(z,\rho(z,B))\subset K$ for all $z\in K$,
then $B$ is called a \textbf{border} for $K$. The unit circle $\mathsf{U}$ is
clearly a border for the unit disc $\overline{B}(0,1)$. Borders are important
for the constructive \textbf{maximum modulus principle}, which we state in an
extended form:

\begin{quote}
\emph{If }$B$\emph{\ is a border for the compact set }$K\subset\mathbb{C}%
$\emph{, and }$f:K\rightarrow\mathbb{C}$\emph{\ is a differentiable function,
then }$\left\Vert f\right\Vert _{K}=\left\Vert f\right\Vert _{B}$\emph{. If
also }$m(f,B)>0$\emph{, then either }$m(f,K)=m(f,B)$\emph{\ or }$m(f,K)=0$
\cite[(5.5.2) and (5.5.3)]{BB}.
\end{quote}

From Proposition \ref{oct29p1} we obtain:

\begin{corollary}
\label{nov02p2}If $\gamma$ is a Jordan curve, then $\mathsf{car}(\gamma)$ is a
border for $\overline{\mathsf{in}(\gamma)}$.
\end{corollary}%

\medskip

The next lemma will lead us to a strengthening of the Corollary to Proposition
4 of \cite{JCTthm}.

\begin{lemma}
\label{aug16l1}A pointwise continuous function $f:\left[  0,1\right]
\rightarrow\left\{  0,1\right\}  $ is constant.
\end{lemma}

\begin{proof}
Fix $x\in\left[  0,1\right]  $ and suppose that $f(x)=1-f(0)$. Then $x>0$
since $f$ is continuous at $0$.$\ $An interval-halving argument enables us to
construct sequences $\left(  a_{n}\right)  _{n\geq1}$ and $\left(
b_{n}\right)  _{n\geq1}$ in $\left[  0,x\right]  $ such that for each $n$,
$f(a_{n})=f(0)$, $f(b_{n})=1-f(0)$, $\left\vert a_{n+1}-a_{n}\right\vert
\leq2^{-n}x$, $\left\vert b_{n+1}-b_{n}\right\vert \leq2^{-n}x$, and
$b_{n}-a_{n}=2^{-n-1}x$. It follows that $\left(  a_{n}\right)  _{n\geq1}$ and
$\left(  b_{n}\right)  _{n\geq1}$ are Cauchy sequences in $[0,x]$ and so
converge to limits $a_{\infty},b_{\infty}\in\left[  0,x\right]  $
respectively; clearly $a_{\infty}=b_{\infty}$. By the pointwise continuity of
$f$, $f(a_{\infty})=\lim_{n\rightarrow\infty}f(a_{n})=f(0)$ and $f(a_{\infty
})=\lim_{n\rightarrow\infty}f(b_{n})=1-f(0)$, which is absurd. Hence
$f(x)\neq1-f(0)$ and therefore $f(x)=f(0)$.%
\hfill

\end{proof}%

\medskip

\begin{proposition}
\label{aug16p2}Let $\gamma:\left[  a,b\right]  \rightarrow\mathbb{C}$ be a
Jordan curve, and $\sigma$ a pointwise continuous mapping of a proper compact
interval $\left[  a,b\right]  $ into $-\mathsf{car}(\gamma)$. Then
$\mathsf{ind}(\sigma(t);\gamma)=\mathsf{ind}(\sigma(a);\gamma)$ for each
$t\in\left[  a,b\right]  $.
\end{proposition}

\begin{proof}
We may assume that $a=0$ and $b=1$. Given $t\in\left[  0,1\right]  $, let $r=$
$\rho(\sigma(t),\mathsf{car}(\gamma))>0$ and choose $\delta>0$ such that
$\left\vert \sigma(t^{\prime})-\sigma(t)\right\vert <r\ $whenever $t^{\prime
}\in\left[  0,1\right]  $ and $\left\vert t-t^{\prime}\right\vert <\delta$.
For such $t^{\prime}$ we have $\sigma(t^{\prime})\subset B(\sigma(t),r)$ and
therefore, by Proposition \ref{oct15p1}, $\mathsf{ind}(\sigma(t^{\prime
});\gamma)=\mathsf{ind}(\sigma(t),\gamma)$. Hence $t\rightsquigarrow
\mathsf{ind}(\sigma(t);\gamma)$ is a pointwise continuous mapping of $\left[
0,1\right]  $ into $\left\{  0,1\right\}  $. It remains to apply Lemma
\ref{aug16l1}.%
\hfill

\end{proof}%

\medskip

To see the constructive significance of this result, we define a \textbf{link
joining points}\emph{\ }$z,z^{\prime}\in\mathbb{C}$ to be a \emph{pointwise}
continuous mapping $\sigma\ $of a proper compact interval $\left[  a,b\right]
$ into $\mathbb{C}$ such that $\sigma(a)=z$ and $\sigma(b)=z^{\prime}%
$.\ Proposition \ref{aug16p2} shows that if two points of $\mathbb{C}$ are
joined by a link $\sigma\ $each of whose points lies off $\gamma$, then all
points of $\sigma$ belong to the same connected component of $-\mathsf{car}%
(\gamma)$. In contrast, JCT gives the same conclusion under the stronger
hypothesis that $\sigma$ is uniformly continuous and bounded away from
$\gamma$.\footnote{%
\normalfont\sf
Classically, if $\sigma\lbrack0,1]\subset-\mathsf{car}(\gamma)$, then $\sigma$
is bounded away from $\gamma$, since the function $t\rightsquigarrow
\rho(\sigma(t),\mathsf{car}(\gamma))$ is uniformly continuous on $[0,1]$ and
so has positive infimum. Constructively, we cannot use this argument, since,
in the recursive interpretation of Bishop's constructive mathematics, there is
an example of a uniformly continuous, positive-valued mapping on $\left[
0,1\right]  $ whose infimum is $0$; see Chapter 6 of \cite{BR}, and in
particular Corollary (2.9) of that chapter.}

\section{Crossing a Jordan curve}

Let $L_{1},L_{2}\ $be lines in the complex plane. We say that $L_{1}%
\ $\textbf{intersects} $L_{2}$ \textbf{uniquely} at the point $\zeta\in
L_{1}\cap L_{2}$ if $\rho(z,L_{2})>0$ for each $z\in L_{1}-\left\{
\zeta\right\}  $. This is the case if and only if the (supplementary) angles
between $L_{1}$ and $L_{2}$ are positive$.$

\begin{lemma}
\label{dec27l1}Let $L$ be a line in $\mathbb{C}$, and $\zeta\in\mathbb{C}$.
Then there exists $a\in L$ such that if $\zeta\neq a$, then $\rho(\zeta,L)>0$.
\end{lemma}

\begin{proof}
Rotating and translating if necessary, we may assume that $L$ is the $x$-axis.
In that case, setting $a=(\operatorname{Re}\zeta,0)$, we see that $a\in L$ and
that if $\zeta\neq a$, then $\rho(\zeta,L)=\left\vert \operatorname{Im}%
\zeta\right\vert >0$.%
\hfill

\end{proof}

\begin{proposition}
\label{dec27p1}Let $L_{1},L_{2}$ be lines in $\mathbb{C}$ such that $L_{1}$
intersects $L_{2}$ uniquely at $\zeta\in L_{1}\cap L_{2}$. Then $L_{2}$
intersects $L_{1}$ uniquely at $\zeta$.
\end{proposition}

\begin{proof}
Let $z_{2}\in L_{2}-\left\{  \zeta\right\}  $. By Lemma \ref{dec27l1}, there
exists $z_{1}\in L_{1}$ such that if $z_{2}\neq z_{1}$, then $\rho(z_{2}%
,L_{1})>0$. Either $z_{1}\neq z_{2}$ or $z_{1}\neq\zeta$; in the latter event,
since $L_{1}$ intersects $L_{2}$ uniquely at $\zeta$, we have $\rho
(z_{1},L_{2})>0$, so $z_{1}\neq z_{2}$. Thus in either case, $z_{2}\neq z_{1}$
and therefore $\rho(z_{2},L_{1})>0$. Hence $L_{2}$ intersects $L_{1}$ uniquely
at $\zeta$.%
\hfill

\end{proof}%

\medskip

In view of the foregoing proposition, we say that \textbf{lines }$L_{1}%
$\textbf{\ and }$L_{2}$\textbf{\ in }$\mathbb{C}$\textbf{\ intersect uniquely
at} $\zeta\in L_{1}\cap L_{2}$ if either intersects the other uniquely at
$\zeta$.

Our next aim is to establish a property of piecewise smooth Jordan curves
that, though intuitively clear, seems to require a surprisingly delicate
constructive proof:

\begin{proposition}
\label{jun04p1}Let $\gamma:\left[  a,b\right]  \rightarrow\mathbb{C}%
\mathbf{\ }$be a piecewise differentiable Jordan curve, let $\zeta_{0}%
\equiv\gamma(t_{0})$ be a smooth point of $\gamma$, and let $T$ be the
tangent, and $N$ the normal, to $\gamma$ at $\zeta_{0}$. Let $0<\theta<\pi/2$,
and let $L_{1},L_{2}$ be the two lines that intersect at $\zeta_{0}$ and make
an angle $\theta$ with $T$. Let $K$ be the cone with vertex $\zeta_{0}$, sides
$L_{1}$ and $L_{2}$, and axis $N$. Then there exist $r,\delta>0$ such that

\begin{enumerate}
\item[\emph{(i)}] the only points of $\gamma$ in the ball $B\equiv\overline
{B}(\zeta_{0},2\delta)$ are those in $\gamma\lbrack t_{0}-r,t_{0}+r]\cap B$;

\item[\emph{(ii)}] if $z\in K^{\circ}$ and $\left\vert z-\zeta_{0}\right\vert
\leq\delta$, then $\rho(z,\mathsf{car}(\gamma))\geq\rho(z,-K)>0$; and

\item[\emph{(iii)}] if $z,z^{\prime}$ belong to $K^{\circ}\cap\overline
{B}(\zeta_{0},\delta)$ and are on opposite sides of $T$, then $\mathsf{ind}%
(z;\gamma)=1-\mathsf{ind}(z^{\prime};\gamma)$.
\end{enumerate}
\end{proposition}%

\medskip
\noindent
In preparation for our proof of Proposition \ref{jun04p1} we have several lemmas.

\begin{lemma}
\label{feb01l1}Let $\gamma:\left[  a,b\right]  \rightarrow\mathbb{C}$ be a
Jordan curve, $t_{0}\in\left(  a,b\right)  $, and $0<r<\min\left\{
t_{0}-a,b-t_{0}\right\}  $. Then there exists $\alpha\in\left(  0,r\right)  $
such that if $t\in\left[  a,b\right]  $ and $\left\vert \gamma(t)-\gamma
(t_{0})\right\vert \leq\alpha$, then $t\in\left[  t_{0}-r,t_{0}+r\right]  $.
\end{lemma}

\begin{proof}
Let%
\[
m=\tfrac{1}{2}\left\vert \gamma(a)-\gamma(t_{0})\right\vert =\tfrac{1}%
{2}\left\vert \gamma(b)-\gamma(t_{0})\right\vert ,
\]
which is positive by \textsf{J1}. By the continuity of $\gamma$, there exists
$s$ with $a+s<t_{0}<b-s$ such that if $t\in\left[  a,b\right]  $ and either
$\left\vert t-a\right\vert <2s$ or $\left\vert t-b\right\vert <2s$, then
$\left\vert \gamma(t)-\gamma(t_{0})\right\vert >m$. By J2, there exists
$\alpha\in\left(  0,m\right)  $ such that if $t,t^{\prime}\in\lbrack a+s,b-s]$
and $\left\vert \gamma(t)-\gamma(t^{\prime})\right\vert \leq\alpha$, then
$\left\vert t-t^{\prime}\right\vert \leq r$. Consider any $t\in\left[
a,b\right]  $ such that $\left\vert \gamma(t)-\gamma(t_{0})\right\vert
\leq\alpha$. We have $\left\vert t-a\right\vert <2s$ or $\left\vert
t-b\right\vert <2s$ or $t\in\lbrack a+s,b-s]$. Since each of the first two
cases yields the contradiction $\left\vert \gamma(t)-\gamma(t_{0})\right\vert
>m>\alpha$, we must have $t\in\lbrack a+s,b-s]$ and therefore, by our choice
of $\alpha$, $\left\vert t-t_{0}\right\vert \leq r$.%
\hfill

\end{proof}

\begin{lemma}
\label{may23l2}Let $\gamma:\left[  a,b\right]  \rightarrow\mathbb{C}$ be a
differentiable path such that $\gamma(a)\neq\gamma(b)$ and $0<m\equiv
\inf\left\{  \left\vert \gamma^{\prime}(t)\right\vert :t\in\left[  a,b\right]
\right\}  $. Then for each $\varepsilon\in(0,\pi/2)$ there exists $t\in\left(
a,b\right)  $ such that the acute angle between $\gamma^{\prime}(t)$ and the
vector $\overrightarrow{\gamma(a)\gamma(b)}\ $is less than $\varepsilon$.
\end{lemma}

\begin{proof}
Rotating, we may assume that $\operatorname{Im}\gamma(a)=0=\operatorname{Im}%
\gamma(b)$, so that $\overrightarrow{\gamma(a)\gamma(b)}$ lies along the
$x$-axis$.$Given $\varepsilon\in\left(  0,\pi/2\right)  $ and applying Rolle's
theorem \cite[page 47, (5.5)]{BB} to the differentiable function
$\operatorname{Im}\gamma$ on $\left[  a,b\right]  $, we obtain $t\in\left[
a,b\right]  $ such that $\left\vert \operatorname{Im}\gamma^{\prime
}(t)\right\vert <m\sin\varepsilon$. If $\theta$ is the acute angle between
$\gamma^{\prime}(t)$ and the $x$-axis, then
\[
\sin\theta=\frac{\left\vert \operatorname{Im}\gamma^{\prime}(t)\right\vert
}{\left\vert \gamma^{\prime}(t)\right\vert }\leq\frac{\left\vert
\operatorname{Im}\gamma^{\prime}(t)\right\vert }{m}<\sin\varepsilon,
\]
so $\theta<\varepsilon$.%
\hfill

\end{proof}

\begin{lemma}
\label{feb12l1}Let $\xi_{1},\eta_{1},\eta_{2},\xi_{2}$ be distinct points of a
line $N$ in $\mathbf{C}$, such that $\eta_{1}\in\left(  \xi_{1},\eta
_{2}\right)  $ and $\eta_{2}\in\left(  \eta_{1},\xi_{2}\right)  $. Let
$0<\varepsilon<\frac{\pi}{2}$ and $\delta=\frac{1}{2}\left\vert \eta_{1}%
-\eta_{2}\right\vert \sin\varepsilon$. Then if $w_{1},w_{2}\in\mathbf{C}$ and
$\rho(w_{k},[\xi_{k},\eta_{k}])\leq\delta$ $(k=1,2)$, the acute angle between
the vectors $\overrightarrow{w_{1}w_{2}}$ and $N$ is at most $\varepsilon$.
\end{lemma}

\begin{proof}
Let $0<r<\frac{1}{2}\left\vert \eta_{1}-\eta_{2}\right\vert $ and let $C_{k}$
be the circle with centre $\eta_{k}$ and radius $r$. Elementary Euclidean
geometry shows that for $\rho(w_{k},[\xi_{k},\eta_{k}])\leq r$ $(k=1,2)$, the
acute angle between $\overrightarrow{w_{1}w_{2}}$ and $N$ is largest when
$w_{k}$ is the point of contact with $C_{k}$ of a common tangent $T$ to
$C_{1}$ and $C_{2}$ that cuts $\left[  \eta_{1},\eta_{2}\right]  $ at
$\frac{1}{2}\left(  \eta_{1}+\eta_{2}\right)  $. By elementary trigonometry,
this largest acute angle is $\sin^{-1}\frac{2r}{\left\vert \eta_{1}-\eta
_{2}\right\vert }$. It follows that if $\rho(w_{k},[\xi_{k},\eta_{k}])<\delta
$, and $\theta$ is the acute angle between $\overrightarrow{w_{1}w_{2}}$ and
$N$, then $\theta$ $<\sin^{-1}\frac{2\delta}{\left\vert \eta_{1}-\eta
_{2}\right\vert }=\varepsilon$.%
\hfill

\end{proof}

\begin{lemma}
\label{nov04l1}Let $L_{1},L_{2}$ be two lines in $\mathbb{C}$ that intersect
uniquely. Then the two closed cones $C_{0},C_{1}$ bounded by $L_{1}$ and
$L_{2}$ are located, $C_{k}^{\circ}=C_{k}-\left(  L_{1}\cup L_{2}\right)  $,
$C_{0}^{\circ}\cup C_{1}^{\circ}=-(L_{1}\cup L_{2})$, and $C_{k}^{\circ
}=-C_{1-k}=-(-C_{k})$ $(k=0,1)$.
\end{lemma}

\begin{proof}
It is elementary that $C_{k}^{\circ}=C_{k}-\left(  L_{1}\cup L_{2}\right)  $,
that $C_{0}^{\circ}\cup C_{1}^{\circ}=-(L_{1}\cup L_{2})$, that $C_{k}^{\circ
}$ is dense in $C_{k}$, that $C_{0}^{\circ}\cup C_{1}^{\circ}$ is dense in
$\mathbb{C}$, that $C_{0}^{\circ}\cap C_{1}^{\circ}=\varnothing$, and that if
$z\in C_{k}^{\circ}$, then $\rho(z,C_{1-k}^{\circ})=\min\{\rho(z,L_{1}%
),\rho(z,L_{2})\}>0$. Hence $\rho(z,C_{1-k})$, equal to $\rho(z,C_{1-k}%
^{\circ})$, exists for all $z$ in the dense subset $C_{0}^{\circ}\cup
C_{1}^{\circ}$ of $\mathbb{C}$, so $C_{1-k}$ is located. Also, $C_{k}^{\circ
}\subset-C_{1-k}$. On the other hand, if $z\in-C_{1-k}$, then $z\in-(L_{1}\cup
L_{2})=C_{0}^{\circ}\cup C_{1}^{\circ}$; since $z\notin C_{1-k}^{\circ}$, we
must have $z\in C_{k}^{\circ}$. Hence $-C_{1-k}\subset C_{k}^{\circ}$ and
therefore $C_{k}^{\circ}=-C_{1-k}$. Finally, since $C_{1-k}^{\circ}$ is dense
in $C_{1-k}$, we have $C_{k}^{\circ}=-C_{1-k}^{\circ}=-(-C_{k})$.%
\hfill

\end{proof}%

\medskip

This brings us to the \textbf{proof\ of\ Proposition }\ref{jun04p1}.%

\medskip

\begin{proof}
Rotating if necessary, we may assume that $\operatorname{Im}\gamma^{\prime
}(t_{0})=0$, so that $\gamma^{\prime}(t_{0})$ is parallel to the $x$-axis. Let%
\begin{align*}
K_{1} &  =\left\{  z\in K-\{\zeta_{0}\}:\theta\leq\mathsf{Arg}\ z\leq
\pi-\theta\right\}  ,\\
K_{2} &  =\left\{  z\in K-\{\zeta_{0}\}:-\pi+\theta\leq\mathsf{Arg}%
\ z\leq-\theta\right\}  ,
\end{align*}
where $\mathsf{Arg}$ denotes the principal value of the argument---that is,
the value in $(-\pi,\pi]$. Then $K-\{\zeta_{0}\}=K_{1}\cup K_{2}$, and $K_{1}$
(respectively, $K_{2})$ is that part of $K$ above (respectively, below) the
tangent $T$. By Lemma \ref{nov04l1}, $K_{j}^{\circ}=K_{j}-\left(  L_{1}\cup
L_{2}\right)  $, from which we see that
\begin{align*}
K_{1}^{\circ} &  =\left\{  z\in K-\{\zeta_{0}\}:\theta<\mathsf{Arg}%
\ z<\pi-\theta\right\}  ,\\
K_{2}^{\circ} &  =\left\{  z\in K-\{\zeta_{0}\}:-\pi+\theta<\mathsf{Arg}%
\ z<-\theta\right\}
\end{align*}
Again by Lemma \ref{nov04l1}, $-K$ is the interior of the closed cone with
axis $T$ and sides $L_{1},L_{2}$, and both $K$ and $-K$ are located. Using the
continuity of $\gamma$ and $\gamma^{\prime}$, now choose $r$ such that
$0<r<\min\left\{  t_{0}-a,b-t_{0}\right\}  $, $\gamma$ is differentiable on
$[t_{0}-r,t_{0}+r]$, and for each $t\in\lbrack t_{0}-r,t_{0}+r]$,

\begin{itemize}
\item[(a$^{\ast}$)] if $\gamma(t)\neq\gamma(t_{0})$, then $\gamma(t)\in-K$,

\item[(b$^{\ast}$)] $\left\vert \gamma^{\prime}(t)\right\vert \geq\frac{1}%
{2}\left\vert \gamma^{\prime}(t_{0})\right\vert $, and

\item[(c$^{\ast}$)] the acute angle between $\gamma^{\prime}(t)$ and $T$ is
less than$\ \theta$.
\end{itemize}%

\noindent
By Lemma \ref{feb01l1}, there exists $\delta\in(0,r)$ such that if
$t\in\left[  a,b\right]  $ and $\left\vert \gamma(t)-\gamma(t_{0})\right\vert
\leq2\delta$, then $t\in\left[  t_{0}-r,t_{0}+r\right]  $. Let $z_{1}%
=\zeta_{0}+i\delta$ and $z_{2}=\zeta_{0}-i\delta$ (so $z_{1}$ lies at a
distance $\delta$ from $\zeta_{0}$ and directly above it, and $z_{2}$ lies at
the same distance from $\zeta_{0}$ and directly below it). For $j=1,2$ let
\[
S_{j}\equiv K_{j}^{\circ}\cap\overline{B}(\zeta_{0},\delta)\subset K^{\circ}.
\]
By Lemma \ref{nov04l1}, $S_{j}\subset-(-K)$. If $z\in S_{j}$, then since
$\gamma(t_{0})\in\overline{-K}$,%
\[
0<\rho(z,-K)\leq\left\vert z-\zeta_{0}\right\vert \leq\delta.
\]
Suppose there exists $t\in\left[  a,b\right]  $ such that $\left\vert
z-\gamma(t)\right\vert <\rho(z,-K)$; then $\gamma(t)\notin-K$. Also,
$\left\vert z-\gamma(t)\right\vert <\left\vert z-\gamma(t_{0})\right\vert
\ $and therefore $\gamma(t)\neq\gamma(t_{0})$. But%
\[
\left\vert \gamma(t)-\gamma(t_{0})\right\vert \leq\left\vert z-\gamma
(t)\right\vert +\left\vert z-\zeta_{0}\right\vert <2\left\vert z-\gamma
(t_{0})\right\vert \leq2\delta,
\]
so $t\in\lbrack t_{0}-r,t_{0}+r]$,$\ $and therefore, by (a$^{\ast}$),
$\gamma(t)\in-K$\thinspace---a contradiction. It follows that $\rho
(z\mathsf{,car}(\gamma))\geq\rho(z,-K)$. This proves (ii) and shows that
$S_{j}\subset-\mathsf{car}(\gamma)$. Since $S_{j}$ is convex and therefore
path connected, Proposition \ref{aug16p2} shows us that $\mathsf{ind}%
(z;\gamma)=\mathsf{ind}(z_{k};\gamma)$ for each $z\in S_{j}$. To prove (iii)
and hence complete the proof of our proposition, it will therefore suffice to
show that one of the points $z_{1},z_{2}$ has index $1$ with respect to
$\gamma$, and the other has index $0$.

To that end, suppose that both $z_{1}$ and $z_{2}$ have index $1$, in which
case $S_{1}\cup S_{2}\subset\mathsf{in}(\gamma)$. Using Theorem 2 of
\cite{JCTthm}, choose $w\in\mathsf{out}(\gamma)$ such that $\left\vert
w-\gamma(t_{0})\right\vert <\frac{\delta}{2}\cos\theta$; then $0<d\equiv
\rho(w,\mathsf{car}(\gamma))<\frac{\delta}{2}$. By Proposition \ref{oct15p1},
$B_{1}\equiv B(w,d)\subset\mathsf{out}(\gamma)$, so by Corollary
\ref{aug02c1}, $S_{1}\cup S_{2}\subset-B_{1}$. If $z\in B_{1}$, then since
$z\in-\mathsf{car}(\gamma)$, we have%
\[
0<\left\vert z-\zeta_{0}\right\vert \leq\left\vert z-w\right\vert +\left\vert
w-\zeta_{0}\right\vert <d+\frac{\delta}{2}<\delta;
\]
if also $\rho(z,-K)>0$, then by Lemma \ref{nov04l1},%
\[
z\in K^{\circ}\cap\overline{B}(\zeta_{0},\delta)=S_{1}\cup S_{2}\subset-B_{1},
\]
which is absurd. Hence $\rho(z,-K)=0$ and $z\in\overline{\left(  -K\right)  }
$. Noting that $B_{1}$ is open, we conclude that $z\subset\overline{\left(
-K\right)  }^{\circ}=-K$.\ If, moreover, $z\in B_{1}\cap N$, then (since
$z\neq\zeta_{0}$) $z\in K_{1}^{\circ}\cup K_{2}^{\circ}$, so by Lemma
\ref{nov04l1}, $\rho(z,-K)>0$, which we have just shown to be absurd. Thus
$B_{1}\cap N=\varnothing$, $\rho(w,N)\geq d$, and therefore the line $N_{w}$
through $w$ and parallel to $N$ is bounded away from $N$. Let $\xi_{j}$ be the
point where $N_{w}$ intersects $L_{j}$. Elementary trigonometry shows that%
\[
0<\left\vert \xi_{j}-\zeta_{0}\right\vert =\frac{\rho(\xi_{j},N)}{\cos\theta
}=\frac{\rho(w,N)}{\cos\theta}\leq\frac{\left\vert w-\zeta_{0}\right\vert
}{\cos\theta}<\frac{\delta}{2}.
\]
For all $z\in N_{w}$ on the opposite side of $L_{j}$ from $w$ and such that
$\left\vert z-\xi_{j}\right\vert <\frac{\delta}{2}$, we have $\left\vert
z-\zeta_{0}\right\vert <\delta$ and therefore $z\in S_{j}$; it readily follows
from this that $\xi_{j}\in\overline{\mathsf{in}(\gamma)}$. Thus by Proposition
\ref{mar6p2}, $\left\vert \xi_{j}-w\right\vert \geq\rho(w,\overline
{\mathsf{in}(\gamma)})=d$. Let%
\[
\lambda_{j}=\frac{d}{\left\vert \xi_{j}-w\right\vert }\text{ \ and \ }\eta
_{j}=\left(  1-\lambda_{j}\right)  w+\lambda_{j}\xi_{j}.
\]
Then $\eta_{j}\in\left[  w,\xi_{j}\right]  \cap N_{w},$ $\left\vert \eta
_{j}-w\right\vert =d,\ \left\vert \eta_{1}-\eta_{2}\right\vert =2d,$
$\left\vert \eta_{j}-\zeta_{0}\right\vert \leq d+\frac{\delta}{2}\,<\delta,$
and $\eta_{j}\in\overline{\mathsf{out}}(\gamma)$. By Lemma \ref{feb12l1},
there exists $\delta_{1}\in\left(  0,d\right)  $ such that for $k\in\left\{
1,2\right\}  $, if $\rho(w_{k},\left[  \eta_{k},\xi_{k}\right]  )<\delta_{1}$,
then the acute angle between the vectors $\overrightarrow{w_{1}w_{2}}$ and
$\overrightarrow{z_{1}z_{2}}$ is less than$\ \frac{1}{2}\left(  \frac{\pi}%
{2}-\theta\right)  $. Since, by Proposition \ref{mar7p1}, $\mathsf{lin}%
(\eta_{j},\xi_{j})$ comes arbitrarily close to $\gamma$, we can choose
$\tau_{j}\in\left[  a,b\right]  $ and $\alpha_{j}\in\left[  0,1\right]  $ such
that $\left\vert \gamma(\tau_{j})-p_{j}\right\vert <\delta_{1}$, where
$p_{j}=(1-\alpha_{j})\eta_{j}+\alpha_{j}\xi_{j}\in\left[  \eta_{j},\xi
_{j}\right]  $. Since $p_{j}$ lies on $N_{w}\ $between $\eta_{j}$ and $\xi
_{j}$, $w$ lies on $N_{w}\ $between $p_{1}$ and $p_{2}$ and therefore%
\[
\left\vert p_{1}-p_{2}\right\vert =\left\vert p_{1}-w\right\vert +\left\vert
p_{2}-w\right\vert \geq\left\vert \eta_{1}-w\right\vert +\left\vert \eta
_{2}-w\right\vert =2d.
\]
Hence%
\[
\left\vert \gamma(\tau_{1})-\gamma(\tau_{2})\right\vert \geq\left\vert
p_{1}-p_{2}\right\vert -\left\vert \gamma(\tau_{1})-p_{1}\right\vert
-\left\vert \gamma(\tau_{2})-p_{2}\right\vert >2d-2\delta_{1}>0.
\]
By our choice of $\delta_{1}$, the acute angle between the vectors
$\overrightarrow{\gamma(\tau_{1})\gamma(\tau_{2})}$ and $N_{w}$ is less
than$\ \frac{1}{2}\left(  \frac{\pi}{2}-\theta\right)  $. On the other hand,
\begin{align*}
\left\vert \gamma(\tau_{j})-\gamma(t_{0})\right\vert  &  \leq\left\vert
\gamma(t)-p_{j}\right\vert +\left\vert p_{j}-\zeta_{0}\right\vert \\
&  <\delta_{1}+\left(  1-\alpha_{j}\right)  \left\vert \eta_{j}-\zeta
_{0}\right\vert +\alpha_{j}\left\vert \xi_{j}-\zeta_{0}\right\vert \\
&  <d+\left(  1-\alpha_{j}\right)  \delta+\alpha_{j}\frac{\delta}{2}\\
&  <\frac{\delta}{2}+\delta+\frac{\delta}{2}=2\delta\text{,}%
\end{align*}
and therefore $t_{0}-r\leq\tau_{j}\leq t_{0}+r$. By \textsf{(b}$^{\ast}%
$\textsf{)},
\[
\inf\{\left\vert \gamma^{\prime}(t)\right\vert :t\in\left[  \tau_{1},\tau
_{2}\right]  \}\geq\tfrac{1}{2}\left\vert \gamma^{\prime}(t_{0})\right\vert
\]
and therefore, by Lemma \ref{may23l2}, there exists $c$ between $\tau_{1}$ and
$\tau_{2}$ such that the acute angle between $\gamma^{\prime}(c)$ and the
vector $\overrightarrow{\gamma(\tau_{1})\gamma(\tau_{2})}$ is less than
$\frac{1}{2}(\frac{\pi}{2}-\theta)$; then the acute angle between
$\gamma^{\prime}(c)$ and $N_{w}$ is less than $\frac{\pi}{2}-\theta$, so the
acute angle between $\gamma^{\prime}(c)$ and $T$ is greater than $\theta$.
Since $c\in\left[  t_{0}-r,t_{0}+r\right]  $, this contradicts \textsf{(c}%
$^{\ast}$\textsf{)}. This final contradiction shows that $z_{1}$ and $z_{2}$
do \emph{not} have the same index $1$. A similar argument shows that they
cannot both have index $0$. Hence one of those points has index $1$ and the
other has index $0$.%
\hfill

\end{proof}

\begin{corollary}
\label{jan24c1}Let $\gamma:\left[  a,b\right]  \rightarrow\mathbb{C}%
\mathbf{\ }$be a piecewise differentiable Jordan curve, let $\zeta_{0}%
\equiv\gamma(t_{0})$ be a smooth point of $\gamma$, and let $T$ be the
tangent, and $N$ the normal, to $\gamma$ at $\zeta_{0}$. Let $L$ be a line
intersecting $T$ uniquely at $\zeta_{0}$. Then there exist $r,\delta>0$ such that

\begin{enumerate}
\item[\emph{(i)}] the only points of $\gamma$ in the ball $B\equiv\overline
{B}(\zeta_{0},2\delta)$ are those in $\gamma\lbrack t_{0}-r,t_{0}+r]\cap B$; and

\item[\emph{(ii)}] there exist points $z_{1},z_{2}\in L$ on opposite sides of
$T$, such that $\left\vert z_{1}-\zeta_{0}\right\vert =\left\vert z_{2}%
-\zeta_{0}\right\vert =\delta$, one of the segments $(\zeta_{0},z_{k}]$ is a
subset of $\mathsf{in}(\gamma)$, and the other is a subset of $\mathsf{out}%
(\gamma)$.
\end{enumerate}
\end{corollary}

\begin{proof}
Since $L$ and $T$ intersect uniquely at $\zeta_{0}$, the supplementary angles
between those lines are both positive. Let $\theta\,$\ be either of them; let
$L_{1},L_{2}$ be the two lines through $\zeta_{0}$ that make an angle
$\frac{1}{2}\min\left\{  \theta,\pi-\theta\right\}  $ with $T$; and let $K$ be
the cone with vertex $\zeta_{0}$, sides $L_{1}$ and $L_{2}$, and axis the
normal to $\gamma$ at $\zeta_{0}$. Then $L-\{\zeta_{0}\}\subset K^{\circ}$.
Choosing $r,\delta>0$ as in Proposition \ref{jun04p1}, it remains to take
$z_{1},z_{2}$ as the two points of $L$ on opposite sides of $T$ and at a
distance $\delta$ from $\zeta_{0}$.%
\hfill

\end{proof}

\begin{corollary}
\label{feb11c1}Under the hypotheses of \emph{Proposition \ref{jun04p1} }let
$r$ and $\delta$ be as in the conclusion thereof, let $z_{1}\in N$ be such
that $\left\vert z_{1}-\zeta_{0}\right\vert =\delta$, and let $\zeta\in
(\zeta_{0},z_{1}]$. Then $\left[  \zeta,z_{1}\right]  \subset\subset
-\mathsf{car}(\gamma)$.
\end{corollary}

\begin{proof}
Let $r=\left\vert \zeta-\zeta_{0}\right\vert \cos\theta$, which is positive.
Elementary trigonometry shows that if $z\in\left[  \zeta,z_{1}\right]  $, then
$\rho(z,-K)=\left\vert z-\zeta_{0}\right\vert \cos\theta\geq r$. whence, by
(ii) of Proposition \ref{jun04p1}, $\rho(z,\mathsf{car}(\gamma))\geq r$. It
follows that $\left[  \zeta,z_{1}\right]  _{r}\subset-\mathsf{car}(\gamma)$.%
\hfill

\end{proof}

\begin{corollary}
\label{aug09c1}Let $\zeta_{1}$ and $\zeta_{2}$ be distinct smooth points on a
piecewise differentiable Jordan curve $\gamma:\left[  a,b\right]
\rightarrow\mathbb{C}$, and let $L$ be the line joining them. Suppose that for
$k=1,2$ the tangent to $\gamma$ at $\zeta_{k}$ intersects $L$ uniquely at
$\zeta_{k}$, and that $\left(  \zeta_{1},\zeta_{2}\right)  \subset
-\mathsf{car}(\gamma)$. Then all points of $\left(  \zeta_{1},\zeta
_{2}\right)  $ have the same index relative to $\gamma$.
\end{corollary}

\begin{proof}
For each $\lambda\in\left[  0,1\right]  $ let $z_{\lambda}=\left(
1-\lambda\right)  \zeta_{1}+\lambda\zeta_{2}$. By Corollary \ref{jan24c1},
there exist $\lambda_{1},\lambda_{2}\ $such that $0<\lambda_{1}<\lambda_{2}%
<1$, $\mathsf{ind}(z;\gamma)=\mathsf{ind}(z_{\lambda_{1}};\gamma)$ for all
$z\in(\zeta_{1},z_{\lambda_{1}}]$, and $\mathsf{ind}(z;\gamma)=\mathsf{ind}%
(z_{\lambda_{2}};\gamma)$ for all $z\in\lbrack z_{\lambda_{2}},\zeta_{2})$.
Let $0<\alpha<\frac{1}{2}\min\left\{  \lambda_{1},1-\lambda_{2}\right\}  $.
Then $z_{\alpha}\in(\zeta_{1},z_{\lambda_{1}}]$ and $z_{1-\alpha}\in\lbrack
z_{\lambda_{2}},\zeta_{2})$. Consider any $z\in\left(  \zeta_{1},\zeta
_{2}\right)  $. We have $z\in(\zeta_{1},z_{\lambda_{1}}]$ or $z\in\left[
z_{\alpha},z_{1-\alpha}\right]  $ or $z\in\lbrack z_{\lambda_{2}},\zeta_{2})$.
In the first case, $\mathsf{ind}(z;\gamma)=\mathsf{ind}(z_{\lambda_{1}}%
;\gamma)=\mathsf{ind}(z_{\alpha};\gamma)$; in the third case, $\mathsf{ind}%
(z;\gamma)=\mathsf{ind}(z_{\lambda_{2}};\gamma)=\mathsf{ind}(z_{1-\alpha
};\gamma)$. In the middle case, $\left[  z_{\alpha},z_{1-\alpha}\right]
\subset(\zeta_{1},\zeta_{2})\subset-\mathsf{car}(\gamma)$, so, taking
$\sigma=\mathsf{lin}(z_{\alpha},z_{1-\alpha})$ in Proposition \ref{aug16p2} we
see that%
\[
\mathsf{ind}(z;\gamma)=\mathsf{ind}(z_{\alpha};\gamma)=\mathsf{ind}%
(z_{1-\alpha};\gamma).
\]
Putting together the three cases, we see that $\mathsf{ind}(z;\gamma
)=\mathsf{ind}(z_{\lambda_{1}};\gamma)$ for all $z\in\left(  \zeta_{1}%
,\zeta_{2}\right)  $.%
\hfill

\end{proof}%

\medskip

The next corollary captures the informal picture of a line passing in and out
of a Jordan curve.

\begin{corollary}
\label{aug08c1}Let $\gamma:\left[  a,b\right]  \rightarrow\mathbb{C}$ be a
piecewise smooth Jordan curve, $\mathbf{v}$ a unit vector in the plane,
$\zeta_{0}\in\mathbb{C}$, and $L$ the line through $\zeta_{0}$ parallel to
$\mathbf{v}$. Let $0=\lambda_{0}<\lambda_{1}<\cdots<\lambda_{N}$, where
$N\geq2$, and for each $k$ let $\zeta_{k}\equiv\zeta_{0}+\lambda_{k}%
\mathbf{v}$. Suppose that $[\zeta_{0},\zeta_{1})\subset-\mathsf{car}(\gamma),$
that $(\zeta_{N-1},\zeta_{N}]\subset-\mathsf{car}(\gamma)$, and that for
$1\leq k\leq N-1$,

\begin{enumerate}
\item[\emph{(i)}] $\zeta_{k}\ $is a smooth point\emph{\ }of $\gamma$,

\item[\emph{(ii)}] $(\zeta_{k},\zeta_{k+1})\subset-\mathsf{car}(\gamma)$, and

\item[\emph{(iii)}] $L$ intersects the tangent to $\gamma$ at $\zeta_{k}$ uniquely.
\end{enumerate}%

\noindent
Then $\mathsf{ind}(\zeta_{N};\gamma)=\mathsf{ind}(\zeta_{0};\gamma)$ if $N$ is
even, and $\mathsf{ind}(\zeta_{N};\gamma)=1-\mathsf{ind}(\zeta_{0};\gamma)$ if
$N$ is odd.
\end{corollary}

\begin{proof}
First observe that by Corollary \ref{jan24c1}, if $1\leq k\leq N-1$, then
there exist points $a_{k}\in\,(\zeta_{k-1},\zeta_{k})$ and $b_{k}\in
\,(\zeta_{k},\zeta_{k+1})$ such that $\mathsf{ind}(z;\gamma)=1-\mathsf{ind}%
(z^{\prime};\gamma)$ whenever $z\in\lbrack a_{k},\zeta_{k})$ and $z^{\prime
}\in(\zeta_{k},b_{k}]$. It follows from Corollary \ref{aug09c1} that
$\mathsf{ind}(z;\gamma)=1-\mathsf{ind}(z^{\prime};\gamma)$ whenever
$z\in\,(\zeta_{k-1},\zeta_{k})$ and $z^{\prime}\in\,(\zeta_{k},\zeta_{k+1})$.
A simple counting argument completes the proof.%
\hfill

\end{proof}

\section{Index and winding number}

Recall that if $\zeta\in\mathbb{C}$ and $\gamma$ is a piecewise differentiable
closed path in $\mathbb{C}-\{\zeta\}$, then%
\[
\mathsf{j}(\gamma,z_{0})\equiv\frac{1}{2\pi i}\int_{\gamma}\frac{\mathsf{d}%
z}{z-\zeta}%
\]
is an integer, the \textbf{winding number} \textbf{of} $\gamma$
\textbf{around}, or \textbf{with respect to}, $\zeta\ $\cite[(5.4.3)]{BB}. The
following lemmas will help us to prove that when $\gamma$ is a piecewise
smooth Jordan curve, the index and winding number of $\gamma$ with respect to
$\zeta$ are equal.

\begin{lemma}
\label{oct30l1}Let $f:[a,b]\rightarrow\mathbb{C}$ be a continuous function
such that $f(a),f(b)$ lie on opposite sides of a line $N\subset\mathbb{C}$.
Then for each $\varepsilon>0$ there exists $x\in(a,b)$ such that
$\rho(f(x),N)<\varepsilon$.
\end{lemma}

\begin{proof}
Rotating if necessary, we may assume that $N$ is the $x$-axis. The real-valued
function $g\equiv\operatorname{Im}f$ is continuous on $\left[  a,b\right]  $,
and $g(a)g(b)<0$. If $\varepsilon>0$, then by the intermediate value
theorem\footnote{%
\normalfont\sf
Note that the classical intermediate value theorem entails LLPO and is
therefore essentially nonconstructive, but with additional hypotheses we can
there are constructive substitutes} \cite[(2.4.8)]{BB}, there exists
$x\in\left[  a,b\right]  $ such that%
\[
\rho(f(x),N))=\left\vert g(x)\right\vert <\min\left\{  \varepsilon,\left\vert
g(a\right\vert ,\left\vert g(b)\right\vert \right\}  \leq\varepsilon.
\]
Then $g(x)\neq g(a)$ and $g(x)\neq g(b)$, so by the continuity of $g$,
$x\in(a,b)$.%
\hfill

\end{proof}

\begin{lemma}
\label{oct26l1}Let $\gamma:\left[  a,b\right]  \rightarrow\mathbb{C}$ be a
piecewise smooth Jordan curve, let $\zeta_{0}\equiv\gamma(t_{0})$ be a smooth
point of $\gamma$, and let $N$ be the normal to $\gamma$ at $\zeta_{0}$. Then
there exists $\delta>0$ such that

\begin{enumerate}
\item[\emph{(i)}] if $t\in\left(  a,b\right)  $ and $0<\left\vert
t-t_{0}\right\vert <\delta$, then $\gamma(t)\in-N$ and is a smooth point; and

\item[\emph{(ii)}] every point of $\gamma(t_{0}-\delta,t_{0})$ lies on the
opposite side of $N$ from every point of $\gamma\left(  t_{0},t_{0}%
+\delta\right)  $.
\end{enumerate}
\end{lemma}

\begin{proof}
There exists $\alpha>0$ such that every point of $\left[  t_{0}-\alpha
,t_{0}+\alpha\right]  $ is smooth. By Corollary \ref{jan24c1}, there exist
$z_{1},z_{2}\in N-\left\{  \zeta_{0}\right\}  $ such that $(\zeta_{0}%
,z_{1}]\subset\mathsf{in}(\gamma)$, $(\zeta_{0},z_{2}]\subset\mathsf{out}%
(\gamma)$, and $\zeta_{0}=\frac{1}{2}(z_{1}+z_{2})$. Let%
\[
\varepsilon=\tfrac{1}{3}\left\vert z_{1}-\zeta_{0}\right\vert =\tfrac{1}%
{3}\left\vert z_{2}-\zeta_{0}\right\vert .
\]
By the continuity of $\gamma$ and $\gamma^{\prime}$, there exists $\delta
\in\left(  0,\min\left\{  \varepsilon,\alpha\right\}  \right)  $ such that

\begin{itemize}
\item[$\bullet$] if $t,t^{\prime}\in\left[  a,b\right]  $ and $\left\vert
t-t^{\prime}\right\vert <\delta$, then $\left\vert \gamma(t)-\gamma(t^{\prime
})\right\vert <\varepsilon$;

\item[$\bullet$] if $t\in\left[  a,b\right]  $ and $\left\vert t-t_{0}%
\right\vert <\delta$, then $\left\vert \gamma^{\prime}(t)\right\vert \geq
\frac{1}{2}\left\vert \gamma^{\prime}(t_{0})\right\vert $ and the acute angle
between $\gamma^{\prime}(t_{0})\ $and $\gamma^{\prime}(t)$ is less than
$\frac{\pi}{4}$;
\end{itemize}%

\noindent
Consider any $t\in(a,b)$ such that $0<\left\vert t-t_{0}\right\vert <\delta$
and hence $\gamma(t)\neq\zeta_{0}$. By Lemma \ref{jan01l1}, there exists
$\zeta^{\prime}\in N$ with%
\[
\left\vert \zeta^{\prime}-\zeta_{0}\right\vert \leq3\left\vert \gamma
(t)-\zeta_{0}\right\vert <3\varepsilon=\left\vert z_{1}-\zeta_{0}\right\vert
\]
such that if $\gamma(t)\neq\zeta^{\prime}$, then $\rho(\gamma(t),N)>0$. Either
$\gamma(t)\neq\zeta^{\prime}$ or $\zeta^{\prime}\neq\zeta_{0}$; in the latter
event, $\zeta^{\prime}\in\lbrack z_{1},\zeta_{0})\cup(\zeta_{0},z_{2}%
]\subset-\mathsf{car}(\gamma)$ and we again have $\gamma(t)\neq\zeta^{\prime}%
$. Thus $\rho(\gamma(t),N)>0$, and therefore the line through $\gamma(t)$ and
parallel to $N$ is bounded away from $N$.

Now fix $t_{2}\in(t_{0},t_{0}+\delta)$, consider any $t_{1}\in\left(
t_{0}-\delta,t_{0}\right)  $, and suppose that $\gamma(t_{1})$ and
$\gamma(t_{2})$ lie on the same side of $N$. Let%
\[
d=\tfrac{1}{2}\min\left\{  \rho(\gamma(t_{1}),N),\rho(\gamma(t_{2}%
),N)\right\}  \text{.}%
\]
For $k=1,2~$let $L_{k}$ be the line through $\gamma(t_{k})$ parallel to
$N$.\ Observe that $\gamma(t_{2})$ and $\gamma(t_{0})$ are on opposite sides
of the line $\Lambda$ parallel to $N$, on the same side of $N$ as
$\gamma(t_{1})$ and $\gamma(t_{2})$, and at a distance $d$ from $N$. By Lemma
\ref{oct30l1}, there exists $s_{2}\in\left(  t_{0},t_{2}\right)  $ such that
$\rho(\gamma(s_{2}),\mathcal{\Lambda})<d\,$; then the line $L$ through
$\gamma(s_{2})$ parallel to $N$ lies strictly between $N$ and $L_{k}\ (k=1,2)
$. Also, $\gamma(t_{1})$ and $\gamma(t_{0})$ are on opposite sides of $L$.
Applying Lemma \ref{oct30l1} again, we obtain $s_{1}\in\left(  t_{1}%
,t_{0}\right)  $ with $\gamma(s_{1})$ is so close to $L$ that the acute angle
between the vector $\overrightarrow{\gamma(s_{1})\gamma(s_{2})}$ and $L$ is
less than $\frac{\pi}{8}$. Since $\gamma(s_{1})\neq\gamma(s_{2})$ and, by our
choice of $\delta$,%
\[
\inf\left\{  \left\vert \gamma^{\prime}(t)\right\vert :t\in\left[
t_{0}-\delta,t_{0}+\delta\right]  \right\}  \geq\tfrac{1}{2}\left\vert
\gamma^{\prime}(t_{0})\right\vert >0,
\]
we can use Lemma \ref{may23l2} to find $\tau\in\left(  s_{1},s_{2}\right)  $
such that the acute angle between $\gamma^{\prime}(\tau)$ and
$\overrightarrow{\gamma(s_{1})\gamma(s_{2})}$ is less than $\frac{\pi}{8}$.
Then the acute angle between $\gamma^{\prime}(\tau)$ and $L$ is less
than$\ \frac{\pi}{4}\,$; so the acute angle between $\gamma^{\prime}(\tau)$
and $\gamma^{\prime}(t_{0})$ is greater than $\frac{\pi}{4}$, which is absurd
since $\left\vert \tau-t_{0}\right\vert <\delta$. This contradiction ensures
that $\gamma(t_{1})$ and $\gamma(t_{2})$ cannot be on the same side of $N\,$;
since $\gamma(t_{1}),\gamma(t_{2})\in-N$, we conclude that they are on
opposite sides of $N$. Since $t_{1}$ is an arbitrary point of $(t_{0}%
-\delta,t_{0})$ we see that every point of $\gamma\left(  t_{0}-\delta
,t_{0}\right)  $ lies on the opposite side of $N$ from $\gamma(t_{2})$. It
follows likewise that every point of $\gamma(t_{0}-\delta,t_{0})$ lies on the
opposite side of $N$ from every point of $\gamma\left(  t_{0},t_{0}%
+\delta\right)  $.%
\hfill

\end{proof}

\begin{lemma}
\label{jul28l1}Let $C$ be a circle with centre $O$ and radius $r>0$, and $A$ a
point such that $OA>r$. Let $T$ be the point of contact of a tangent from $A$
to $C$, and $P$ any point with $OP\leq r$. Then the acute angle $O\widehat{A}%
P$ is less than or equal to the acute angle $O\widehat{A}T$, which equals
$\sin^{-1}(r/OA)$.
\end{lemma}

\begin{proof}
We may assume that $OP=r$. Take $C$ to be the circle $x^{2}+y^{2}=r$;
$A=(0,-a)$, where $a>r$; and $P=(r\cos\theta,r\sin\theta)$, where $0\leq
\theta\leq\pi/2$. Let $\phi$ be the acute angle $O\widehat{A}P$, and apply
elementary calculus.%
\hfill

\end{proof}

\begin{lemma}
\label{oct26l2}Let $\gamma:\left[  a,b\right]  \rightarrow\mathbb{C}$ be a
piecewise smooth Jordan curve, $\gamma(t_{0})$ a smooth point of $\gamma$, $N$
the normal to $\gamma$ at $\gamma(t_{0})$, and $z_{1}\in N-\{\gamma(t_{0})\}$.
Then for each $\varepsilon\in\left(  0,1\right)  $ there exists $\delta>0$
such that if $t\in\left(  0,1\right)  $ and $0<\left\vert t-t_{0}\right\vert
<\delta$, then $\gamma(t_{0}),z_{1}$, and $\gamma(t)$ are the vertices of a
proper triangle; the segments $\left[  z_{1},\gamma(t_{0})\right]  $ and
$\left[  z_{1},\gamma(t)\right]  $ intersect only at $z_{1}$; and the acute
angle between those segments is $<\sin^{-1}\varepsilon$.
\end{lemma}

\begin{proof}
Let $\varepsilon\in\left(  0,1\right)  $ and $B=\overline{B}(\gamma
(t_{0}),\varepsilon\left\vert \gamma(t_{0})-z_{1}\right\vert )$. By Lemma
\ref{oct26l1} and the continuity of $\gamma$, there exists $\delta>0$ such
that if $t\in\left(  a,b\right)  $ and $0<\left\vert t-t_{0}\right\vert
<\delta$, then $\gamma(t)\in-N$ and $\left\vert \gamma(t)-\gamma
(t_{0})\right\vert <\varepsilon\left\vert \gamma(t_{0})-z_{1}\right\vert $. In
that case, since $\gamma(t_{0})\neq z_{1}$ and $\rho(\gamma(t),[\gamma
(t_{0}),z_{1}])\geq\rho(\gamma(t),N)>0$, the points $\gamma(t_{0}),z_{1}$, and
$\gamma(t)$ are the vertices of a proper triangle and the segments $\left[
z_{1},\gamma(t_{0})\right]  $ and $\left[  z_{1},\gamma(t)\right]  $ intersect
only at $z_{1}$. Also, $\left\vert \gamma(t_{0})-z_{1}\right\vert
>\varepsilon\left\vert \gamma(t_{0})-z_{1}\right\vert $ and $\gamma(t)\in B$,
so, by Lemma \ref{jul28l1}, the acute angle between $\left[  z_{1}%
,\gamma(t_{0})\right]  $ and $\left[  z_{1},\gamma(t)\right]  $ is less than
$\sin^{-1}\varepsilon$.%
\hfill

\end{proof}

\begin{lemma}
\label{jan09l1}Let $\gamma_{1},\gamma_{2}:\left[  0,1\right]  \rightarrow
\mathbb{C}$ be piecewise differentiable paths in a compact set $K\subset
\mathbb{C}$, and let $f\ $be differentiable on $K$. Then%
\[
\left\vert \int_{\gamma_{1}}f(z)\mathsf{d}z-\int_{\gamma_{z}}f(z)\mathsf{d}%
z\right\vert \leq\left\Vert f\right\Vert _{K}\left\Vert \gamma_{1}^{\prime
}-\gamma_{2}^{\prime}\right\Vert _{\left[  0,1\right]  }+\left\Vert \gamma
_{2}^{\prime}\right\Vert _{\left[  0,1\right]  }\left\Vert f\circ\gamma
_{1}-f\circ\gamma_{2}\right\Vert _{\left[  0,1\right]  }.
\]

\end{lemma}

\begin{proof}
We have%
\begin{align*}
&  \left\vert \int_{\gamma_{1}}f(z)\mathsf{d}z-\int_{\gamma_{z}}%
f(z)\mathsf{d}z\right\vert \\
&  =\left\vert \int_{0}^{1}f(\gamma_{1}(t))(\gamma_{1}^{\prime}(t)-\gamma
_{2}^{\prime}(t))\mathsf{d}t+\int_{0}^{1}\gamma_{2}^{\prime}(t)(f(\gamma
_{1}(t))-f(\gamma_{2}(t)))\mathsf{d}t\right\vert \\
&  \leq\int_{0}^{1}\left\vert f(\gamma_{1}(t))\right\vert \left\vert
\gamma_{1}^{\prime}(t)-\gamma_{2}^{\prime}(t)\right\vert \mathsf{d}t+\int%
_{0}^{1}\left\vert \gamma_{2}^{\prime}(t)\right\vert \left\vert f(\gamma
_{1}(t))-f(\gamma_{2}(t))\right\vert \mathsf{d}t\\
&  \leq\int_{0}^{1}\left\Vert f\right\Vert _{K}\left\Vert \gamma_{1}^{\prime
}-\gamma_{2}^{\prime}\right\Vert _{\left[  0,1\right]  }\mathsf{d}t+\int%
_{0}^{1}\left\Vert \gamma_{2}^{\prime}\right\Vert _{\left[  0,1\right]
}\left\Vert f\circ\gamma_{1}-f\circ\gamma_{2}\right\Vert _{\left[  0,1\right]
}\mathsf{d}t\\
&  \leq\left\Vert f\right\Vert _{K}\left\Vert \gamma_{1}^{\prime}-\gamma
_{2}^{\prime}\right\Vert _{\left[  0,1\right]  }+\left\Vert \gamma_{2}%
^{\prime}\right\Vert _{\left[  0,1\right]  }\left\Vert f\circ\gamma_{1}%
-f\circ\gamma_{2}\right\Vert _{\left[  0,1\right]  }.
\end{align*}%
\hfill

\end{proof}

\begin{lemma}
\label{jan09l2}Let $K\subset\mathbb{C}$ be compact, let $\gamma_{1},\gamma
_{2}$ be linear paths in $K$, and let $z_{1}\in-K$. Then%
\[
\left\vert \int_{\gamma_{1}}\frac{\mathsf{d}z}{z-z_{1}}-\int_{\gamma_{2}}%
\frac{\mathsf{d}z}{z-z_{1}}\right\vert \leq\frac{\left\Vert \gamma_{1}%
-\gamma_{2}\right\Vert _{\left[  0,1\right]  }}{\rho(z_{1},K)}\left(
2+\frac{\left\vert \gamma_{2}(1)-\gamma_{2}(0)\right\vert }{\rho(z_{1}%
,K)}\right)  .
\]

\end{lemma}

\begin{proof}
Let $f(z)=1/(z-z_{1})\ \ (z\in\mathbb{C}-\{z_{1}\})$. Then $f$ is
differentiable on $K$ and $\left\Vert f\right\Vert _{K}\leq1/\rho(z_{1},K)$.
For each $t\in\left[  0,1\right]  $,%
\begin{align*}
\left\vert f(\gamma_{1}(t))-f(\gamma_{2}(t))\right\vert  &  =\left\vert
\frac{1}{\gamma_{2}(t)-z_{1}}-\frac{1}{\gamma_{1}(t)-z_{1}}\right\vert \\
&  \leq\frac{\left\vert \gamma_{1}(t)-\gamma_{2}(t)\right\vert }{\left\vert
\gamma_{2}(t)-z_{1}\right\vert \left\vert \gamma_{1}(t)-z_{1}\right\vert }\\
&  \leq\frac{\left\Vert \gamma_{1}-\gamma_{2}\right\Vert _{\left[  0,1\right]
}}{\rho(z_{1},K)^{2}}.
\end{align*}
Also, since $\gamma_{k}$ is linear, $\gamma_{k}^{\prime}(t)=\gamma
_{k}(1)-\gamma_{k}(0)$ and therefore%
\begin{align*}
\left\Vert \gamma_{1}^{\prime}(t)-\gamma_{2}^{\prime}(t)\right\Vert _{ \left[
0,1\right]  } &  =\left\vert \gamma_{1}(1)-\gamma_{1}(0)-\gamma_{2}%
(1)+\gamma_{2}(0)\right\vert \\
&  \leq\left\vert \gamma_{1}(1)-\gamma_{2}(1)\right\vert +\left\vert
\gamma_{1}(0)-\gamma_{2}(0)\right\vert \\
&  \leq2\left\Vert \gamma_{1}-\gamma_{2}\right\Vert _{\left[  0,1\right]  }.
\end{align*}
Hence, by Lemma \ref{jan09l1},%
\begin{align*}
\left\vert \int_{\gamma_{1}}\frac{\mathsf{d}z}{z-z_{1}}-\int_{\gamma_{2}}%
\frac{\mathsf{d}z}{z-z_{1}}\right\vert  & \leq\frac{2\left\Vert \gamma
_{1}-\gamma_{2}\right\Vert _{\left[  0,1\right]  }}{\rho(z_{1},K)}%
+\frac{\left\Vert \gamma_{2}^{\prime}\right\Vert _{\left[  0,1\right]
}\left\Vert \gamma_{1}-\gamma_{2}\right\Vert _{\left[  0,1\right]  }}%
{\rho(z_{1},K)^{2}}\\
& =\frac{\left\Vert \gamma_{1}-\gamma_{2}\right\Vert _{\left[  0,1\right]  }%
}{\rho(z_{1},K)}\left(  2+\frac{\left\vert \gamma_{2}(1)-\gamma_{2}%
(0)\right\vert }{\rho(z_{1},K)}\right)  .
\end{align*}
%

\hfill

\end{proof}

Before our main result in this section, we point out that the statement of
Corollary (5.4.5) on page 144 of \cite{BB} is incorrect as it
stands,\footnote{%
\normalfont\sf
The word `closed' should be deleted from the second line of the Corollary and
the first line of its proof.} and should, in our context, be replaced by the following.

\begin{proposition}
\label{oct31p1}Let $\gamma$ be a closed path in $\mathbb{C}$, and $z_{1}%
,z_{2}$ points of $\mathbb{C}$ that can be joined by a piecewise
differentiable\footnote{%
\normalfont\sf
For Bishop the words \emph{piecewise differentiable} would be redundant.} path
in $-\mathsf{car}(\gamma)$. Then $\mathsf{j}(\gamma,z_{1})=\mathsf{j}%
(\gamma,z_{2})$.
\end{proposition}

For completeness we include the following two results.

\begin{lemma}
\label{oct31l2}Let $\gamma:\left[  a,b\right]  \rightarrow\mathbb{C}$ be a
piecewise differentiable closed curve. Then there exists $R>0$ such that
$\mathsf{j}(\gamma,\zeta)=0$ whenever $\left\vert \zeta\right\vert >R$.
\end{lemma}

\begin{proof}
Since $\gamma$ is piecewise differentiable, there exists $M>0$ such that
$\left\vert \gamma\,^{\prime}(t)\right\vert <M$ for each point $t\in\left[
a,b\right]  $ at which $\gamma$ is differentiable. Choose $R>0$ such that if
$\left\vert \zeta\right\vert >R$, then $\rho(\zeta,\mathsf{ca}$\textsf{$r$%
}$(\gamma))>\frac{(b-a)\left(  1+M\right)  }{2\pi}$. For such $\zeta$ we have%

\begin{align*}
\left\vert j(\gamma,\zeta)\right\vert  & \leq\frac{1}{2\pi}\int_{a}^{b}%
\frac{\left\vert \gamma^{\prime}(t)\right\vert }{\left\vert \gamma
(t)-\zeta\right\vert }\mathsf{d}t\\
& \leq\frac{1}{2\pi}\int_{a}^{b}\frac{M}{\rho(\zeta,\mathsf{car}(\gamma
))}\mathsf{d}t\\
& \leq\frac{1}{2\pi}\int_{a}^{b}\frac{2\pi M}{(b-a)(1+M)}\mathsf{d}t\\
& =\frac{M}{M+1}<1\text{,}%
\end{align*}
so, as the winding number is an integer, $\mathsf{j}(\gamma,z_{1})=0$.%
\hfill

\end{proof}

\begin{lemma}
\label{jan06l1}Let $\gamma:\left[  a,b\right]  \rightarrow\mathbb{C}$ be a
piecewise differentiable closed curve that lies in an open ball $B\subset
\mathbb{C}$. Then $\mathsf{j}(\Gamma,\zeta)=0$ for all $\zeta\in-B$.
\end{lemma}

\begin{proof}
Translating if necessary, we may assume that $B=B(0,r)$ where $r>0$. Given
$\zeta\in-B$, choose $R>0$ as in Lemma \ref{oct31l2}. Either $\left\vert
\zeta\right\vert >R$, in which case $\mathsf{j}(\gamma,\zeta)=0$, or else
$\left\vert \zeta\right\vert <2R$. In the latter case, let $\zeta_{1}%
=\frac{4R}{\left\vert \zeta\right\vert }\zeta$; then $\left\vert \zeta
_{1}\right\vert >R$, so $\mathsf{j}(\gamma,\zeta_{1})=0$. Also, for each
$\lambda\in\left[  0,1\right]  $,%
\[
\left\vert \left(  1-\lambda\right)  \zeta+\lambda\zeta_{1}\right\vert
=\left\vert 1+\left(  \frac{4R}{\left\vert \zeta\right\vert }-1\right)
\lambda\right\vert \left\vert \zeta\right\vert \geq(1+\lambda)\left\vert
\zeta\right\vert >r\text{,}%
\]
so $\left[  \zeta,\zeta_{1}\right]  \subset-B$. Since $\mathsf{car}%
(\gamma)\subset\subset B$, it follows that the path $\mathsf{lin}(\zeta
,\zeta_{1})$ is bounded away from $\mathsf{car}(\gamma)$; whence, by
\cite[(5.4.5)]{BB}, $\mathsf{j}(\gamma,\zeta)=\mathsf{j}(\gamma,\zeta_{1})=0$.%
\hfill

\end{proof}

\begin{lemma}
\label{feb06l1}Let $\gamma:\left[  a,b\right]  \rightarrow\mathbb{C}$ is a
piecewise smooth Jordan curve, then $\mathsf{j}(\gamma,\zeta)=0$ for each
$\zeta\in\mathsf{out}(\gamma)$.
\end{lemma}

\begin{proof}
Using Lemma \ref{oct31l2}, choose $R>0$ such that $\overline{\mathsf{in}%
(\gamma)}\subset\subset B(0,R)$ and $\mathsf{j}(\gamma,\zeta)=0$ whenever
$\left\vert \zeta\right\vert >R$. Fix $\zeta_{1}$ with $\left\vert \zeta
_{1}\right\vert >R$; then $\zeta_{1}\in-\overline{\mathsf{in}(\gamma
)}=\mathsf{out}(\gamma)$. If $\zeta\in\mathsf{out}(\gamma)$, then by JCT,
$\zeta$ and $\zeta_{1}$ can be joined by a polygonal path well contained in
$-\mathsf{ca}$\textsf{$r$}$\mathsf{(}\gamma)$, so, by Proposition
\ref{oct31p1} and our choice of $R$, $\mathsf{j}(\gamma,\zeta)=\mathsf{j}%
(\gamma,\zeta_{1})=0=\mathsf{ind}(\zeta,\gamma)$.%
\hfill

\end{proof}%

\medskip

We have now reached our final result.\label{0001202}

\begin{proposition}
\label{oct25p1}If $\zeta\in\mathbb{C}$ and $\gamma\ $is a piecewise smooth
Jordan curve in $\mathbb{C}-\{\zeta\}$, then $\mathsf{j}(\gamma,\zeta
)=\mathsf{ind}(\zeta;\gamma)$.
\end{proposition}

\begin{proof}
We may take the domain of $\gamma$ to be $\left[  0,1\right]  $. The case
$\zeta\in\mathsf{out}(\gamma)$ is covered by Lemma \ref{feb06l1}. To deal with
the case $z\in\mathsf{in}(\gamma)$, fix $t_{0}\in(0,1)$ such that $\zeta
_{0}\equiv\gamma(t_{0})$ is a smooth point of $\gamma$. Rotating if necessary,
we may assume that $\operatorname{Im}\gamma\,^{\prime}(t_{0})=0$, so that the
tangent $T$ to $\gamma$ at $t_{0}$ is horizontal. Let $N$ be the normal to
$\gamma$ at $t_{0}$. By Corollary \ref{jan24c1}, there exists $\delta>0$ such that

\begin{itemize}
\item[(i)] if $z_{1}=\zeta_{0}-i\delta$ and $z_{2}=\zeta_{0}+i\delta$, then,
without loss of generality, $(\zeta_{0},z_{1}]\subset\mathsf{in}(\gamma)$ and
$(\zeta_{0},z_{2}]\subset\mathsf{out}(\gamma)$.
\end{itemize}%

\noindent
By Lemma \ref{oct26l1}, there exists $r$ with $0<r<\delta$ such that

\begin{itemize}
\item[(ii)] if $0<\left\vert t-t_{0}\right\vert <r$, then $\gamma(t)\in-N$;

\item[(iii)] every point of $\gamma(t_{0}-r,t_{0})$ lies on the opposite side
of $N$ from every point of $\gamma\left(  t_{0},t_{0}+r\right)  $.
\end{itemize}%

\noindent
Let%
\[
D=\sup\left\{  \left\vert z-z_{1}\right\vert :z\in\overline{\mathsf{in}%
(\gamma)}\right\}
\]
and $\zeta_{1}\in-B(z_{1},2D)$. Then $\rho(\zeta_{1},\overline{\mathsf{in}%
(\gamma)})>0$, so by Corollary \ref{aug02c1}, $\zeta_{1}\in\mathsf{out}%
(\gamma)$. Now let $d=\rho(z_{1},\mathsf{ca}$\textsf{$r$}$(\gamma))$ and note
that $\delta=\left\vert z_{1}-\gamma(t_{0})\right\vert \geq d$. Let also
\[
\kappa=1+\frac{M}{d}+4\left(  1+\frac{1}{d}\right)  \left(  1+\frac{\delta}%
{d}\right)
\]
and%
\[
C\equiv\left\{  z\in\mathbb{C}:\left\vert z-z_{1}\right\vert =\frac{d}%
{2}\right\}  .
\]
Construct a polygonal path $\mu:\left[  0,1\right]  \rightarrow\mathbb{C}$ in
$\mathsf{out}(\gamma)$ joining $\zeta_{1}$ to $z_{2}$ such that
\begin{align*}
0<\alpha &  \equiv\inf\left\{  \rho(\mu(s),\mathsf{car}(\gamma)):s\in\left[
0,1\right]  \right\}  \\
&  =\inf\left\{  \rho(\mu(s),\overline{\mathsf{in}(\gamma)}):s\in\left[
0,1\right]  \right\}  ,
\end{align*}
the final equality arising from Proposition \ref{mar6p2}. Let $0<\varepsilon
<\min\left\{  \alpha,\pi\right\}  $. Noting that $\rho(\gamma(t_{0}%
),C)<\left\vert \gamma(t_{0})-z_{1}\right\vert <\delta$, and using Lemma
\ref{oct26l2} and the continuity of $t\rightsquigarrow\rho(\gamma(t),\left[
\gamma(t_{0}),z_{1}\right]  )$ on $\left[  0,1\right]  $, construct $r_{1}%
\ $such that $0<r_{1}<\min\left\{  r,\varepsilon/2\right\}  $ and

\begin{itemize}
\item[(iv)] if $0<\left\vert t-t_{0}\right\vert <r_{1}$, then

\begin{itemize}
\item $\left\vert \gamma(t)-\gamma(t_{0})\right\vert <\varepsilon/2$;

\item the points $\gamma(t_{0}),z_{1}$, and $\gamma(t)$ are the vertices of a
proper triangle;

\item the segments $\left[  z_{1},\gamma(t_{0})\right]  $ and $\left[
z_{1},\gamma(t)\right]  $ intersect only at $z_{1}$;

\item the acute angle between those segments is less than $\varepsilon/2$; and

\item $\rho(\gamma(t),C)<\delta$.
\end{itemize}
\end{itemize}%

\noindent
Fix $t_{1},t_{2}$ such that $t_{0}-r_{1}<t_{1}<t_{0}<t_{2}<t_{0}+r_{1}$. By
(iii), $t_{1}$ and $t_{2}$ are on opposite sides of $N$, as are the segments
$[\gamma(t_{1}),z_{1})$ and $[\gamma(t_{2}),z_{1})$. Let $\sigma$ be the
circular path%
\[
\theta\rightsquigarrow z_{1}+\frac{d}{2}\exp\left[  i\left(  \frac{\pi}%
{2}-\theta\right)  \right]  \ \ \ \ \ \left(  0\leq\theta\leq2\pi\right)  .
\]
which is traversed in the \emph{clockwise} direction, is centred at $z_{1}$,
has radius $\frac{d}{2}$, has carrier $C$, and is such that $\sigma
(0)=z_{1}+\frac{id}{2}=\sigma(2\pi)$. Note that by Proposition \ref{oct15p1},
$B(z_{1},d)\subset\mathsf{in}(\gamma)$. For each $k\in\left\{  1,2\right\}  $
there exists $\theta_{k}\in(0,2\pi)$ such that $\sigma(\theta_{k})$ is the
intersection of $C$ and $[\gamma(t_{k}),z_{1}]$. Let $\gamma_{1}%
=\mathsf{lin}(\gamma(t_{1}),\sigma(\theta_{1}))$ and $\gamma_{3}%
=\mathsf{lin}(\gamma(t_{2}),\sigma(\theta_{2}))$; let also $\gamma_{0}$ be the
restriction of $\gamma$ to $\left[  0,t_{1}\right]  $, $\gamma_{4}$ the
restriction of $\gamma$ to $\left[  t_{2},1\right]  $, and $\gamma_{2}$ the
restriction of $\sigma$ to $\left[  \theta_{1},\theta_{2}\right]  $. Then%
\[
\Gamma\equiv\gamma_{0}+\gamma_{1}+\gamma_{2}-\gamma_{3}+\gamma_{4}%
\]
is a piecewise smooth closed path. We divide the remainder of the proof into
five steps.%

\medskip
\noindent
\textbf{Step 1: \ }$\mu$ \emph{is bounded away from} $\Gamma$.%

\smallskip
\noindent
Since $\gamma_{0}$, $\gamma_{2}$, and $\gamma_{4}$ lie in $\overline
{\mathsf{in}(\gamma)}$, $\mu$ is bounded away by $\alpha$ from each of those
paths. On the other hand, if $\lambda\in\left[  0,1\right]  $, $z=\left(
1-\lambda\right)  \gamma(t_{1})+\lambda z_{1}$, and $s\in\left[  0,1\right]
$, then since $\left[  \zeta_{0},z_{1}\right]  \subset\overline{\mathsf{in}%
(\gamma)}$ and $0<\left\vert t_{1}-t_{0}\right\vert <r_{1}$,%
\begin{align*}
\left\vert z-\mu(s)\right\vert  &  \geq\left\vert \mu(s)-\left(
1-\lambda\right)  \zeta_{0}-\lambda z_{1}\right\vert \\
&  \ \ \ \ \ \ \ \ \left.  -\left\vert \left(  1-\lambda\right)  \gamma
(t_{1})+\lambda z_{1}-\left(  1-\lambda\right)  \zeta_{0}-\lambda
z_{1}\right\vert \right. \\
&  \geq\rho(\mu(s),\overline{\mathsf{in}(\gamma)})-\left(  1-\lambda\right)
\left\vert \gamma(t_{1})-\gamma(t_{0})\right\vert \\
&  >\alpha-\frac{\varepsilon}{2}>\frac{\alpha}{2}.
\end{align*}
Hence $\mu$ is bounded away by $\alpha/2$ from $\left[  \gamma(t_{1}%
),z_{1}\right]  $ and therefore from $\gamma_{1}$; likewise, $\mu$ is bounded
away by $\alpha/2$ from $\gamma_{3}$. Putting all this together, we see that
$\mu$ is bounded away by $\alpha/2$ from $\Gamma$.%

\medskip
\noindent
\textbf{Step 2}: \ $\mathsf{lin}(z_{2},z_{1})$ \emph{is bounded away from}
$\Gamma$.%

\smallskip
\noindent
First note that $\gamma(t_{k}),z_{1}$, and $\zeta_{0}$ are vertices of a
proper triangle, by (iv), that $0<\theta_{k}<\frac{\varepsilon}{2}<\frac{\pi
}{2}$, and that $\sigma(\theta_{k})\in\left(  z_{1},\gamma(t_{k}\right)  )$.
By elementary trigonometry, for each $z\in\left[  \sigma(\theta_{k}%
),\gamma(t_{k})\right]  $ we have%
\[
\rho(z,N)=\left\vert z-z_{1}\right\vert \sin\theta_{k}\geq\tfrac{d}{2}%
\sin\theta_{k}.
\]
It now follows that $\mathsf{lin}(z_{2},z_{1})$ is bounded away from
$\gamma_{1}$ and $\gamma_{3}$ by%
\[
\beta\equiv\min\left\{  \tfrac{d}{2}\sin\theta_{1},\tfrac{d}{2}\sin\theta
_{2}\right\}  .
\]
Since $\rho(\sigma(\theta),[z_{1},z_{2}])$ increases from\ $\tfrac{d}{2}%
\sin\theta_{1}$ to $\frac{d}{2}$ on $\left[  \theta_{1},\frac{\pi}{2}\right]
$, takes the value $\frac{d}{2}$ throughout $\left[  \frac{\pi}{2},\frac{3\pi
}{3}\right]  $, and decreases from $\frac{d}{2}$ to $\tfrac{d}{2}\sin
\theta_{2}$ on $\left[  \frac{3\pi}{2},2\pi\right]  $, we see that
$\mathsf{lin}(z_{2},z_{1})$ is bounded away by $\beta$ from $\gamma_{2}$. To
prove that it is bounded away from $\gamma_{0}$, first observe that by
Proposition \ref{nov22p2}, $\gamma_{0}[0,t_{1}]$ is compact; by \textsf{J1},
$\gamma(t_{0})\neq\gamma(t)$ for each $t\in\left[  0,t_{1}\right]  $; and
hence, by Bishop's Lemma, $\gamma(t_{0})\in-\gamma_{0}\left[  0,t_{1}\right]
$. Likewise, $\gamma(t_{0})\in-\gamma_{4}\left[  t_{2},1\right]  $. Let%
\[
s_{1}=\tfrac{1}{3}\min\left\{  \delta,\rho(\gamma(t_{0}),\mathsf{car}%
(\gamma_{0})),\rho(\gamma(t_{0}),\mathsf{car}(\gamma_{4}))\right\}  >0
\]
and%
\[
\chi_{k}=\left(  1-\frac{s_{1}}{\delta}\right)  \zeta_{0}+\frac{s_{1}}{\delta
}z_{k}\ \ \ (k=1,2).
\]
Then $\chi_{k}\in(\zeta_{0},z_{k}]$ and $\left\vert \chi_{k}-\zeta
_{0}\right\vert =s_{1}$. By Corollary \ref{feb11c1}, there exists $s_{2}$ such
that $0<s_{2}<s_{1}$ and $\left[  \chi_{k},z_{k}\right]  _{s_{2}}%
\subset-\mathsf{car}(\gamma)$\ $(k=1,2).$ Suppose that $t\in\left[
0,t_{1}\right]  $, $z\in$ $\left[  z_{1},z_{2}\right]  $, and $\left\vert
\gamma(t)-z\right\vert <s_{2}$. Either $\left\vert z-\zeta_{0}\right\vert
<2s_{1}$ or $\left\vert z-\zeta_{0}\right\vert >s_{1}$. In the first case,%
\[
\left\vert \gamma(t)-\zeta_{0}\right\vert \leq\left\vert \gamma
(t)-z\right\vert +\left\vert z-\zeta_{0}\right\vert <3s_{1}%
\]
and so%
\[
\rho(\gamma(t),\mathsf{car}(\gamma_{0}))\geq\rho(\gamma(t_{0}),\mathsf{car}%
(\gamma_{0}))-\left\vert \gamma(t)-\zeta_{0}\right\vert >3s_{1}-3s_{1}=0,
\]
which is absurd. In the case $\left\vert z-\zeta_{0}\right\vert >s_{1}$ there
exists $k\in\{1,2\}$ such that $~z\in\left[  \chi_{k},z_{k}\right]  $ and
therefore $\gamma(t)\in\left[  \chi_{k},z_{k}\right]  _{s_{2}}\subset
-\mathsf{car}(\gamma)$, which is equally absurd. It follows that $\rho
(\gamma_{0}(t),\left[  z_{1},z_{2}\right]  )\geq s_{2}$ for all $t\in\left[
0,t_{1}\right]  $; whence$\ \mathsf{lin}(z_{2},z_{1})$ is bounded away
from$\ \gamma_{0}\,$---and, likewise, from $\gamma_{4}$. Putting all this
together, we now see that $\mathsf{lin}(z_{2},z_{1})$ is bounded away from
$\Gamma$.%

\medskip
\noindent
\textbf{Step 3: \ }$\frac{1}{2\pi i}\int_{\Gamma}\frac{\mathsf{d}z}{z-z_{1}%
}=0$%

\smallskip
\noindent
If $z\in$ $\mathsf{car}(\gamma_{1})$, then $\left\vert z-z_{1}\right\vert
\leq\left\vert \gamma(t_{1})-z_{1}\right\vert \leq D$; likewise, if
$z\in\mathsf{car}(\gamma_{3})$, then $\left\vert z-z_{1}\right\vert \leq D$;
and if $z\in\mathsf{car}(\gamma_{0}\cup\gamma_{2}\cup\gamma_{4})$, then
$z\in\overline{\mathsf{in}(\gamma)}$ and again\ $\left\vert z-z_{1}\right\vert
\leq D$. Hence $\mathsf{car}(\Gamma)\subset B(z_{1},D)\subset\subset
B(z_{1},2D)$. Since $\zeta_{1}\in-B(z_{1},2D)$, it follows from Lemma
\ref{jan06l1} that $\mathsf{j}(\Gamma,\zeta_{1})=0$. On the other hand, from
Steps 1 and 2 we conclude that the piecewise linear path $\mu+\mathsf{lin}%
(z_{2},z_{1})$ is bounded away from $\Gamma$. Thus by Proposition
\ref{oct31p1},%
\begin{equation}
\frac{1}{2\pi i}\int_{\Gamma}\frac{\mathsf{d}z}{z-z_{1}}=\mathsf{j(\Gamma
},z_{1})=\mathsf{j}(\Gamma,\zeta_{1})=0.\label{aan15}%
\end{equation}
%

\medskip
\noindent
\textbf{Step 4: \ }$\left(  \int_{\gamma}+\int_{\sigma}\right)  \frac
{\mathsf{d}z}{z-z_{1}}=0$%

\smallskip
\noindent
We now estimate the various components of the integral on the left of
(\ref{aan15}). First,%
\begin{align*}
\left\vert \gamma_{1}(0)-\gamma_{3}(0)\right\vert  &  =\left\vert \gamma
(t_{1})-\gamma(t_{2})\right\vert \\
&  \leq\left\vert \gamma(t_{1})-\gamma(t_{0})\right\vert +\left\vert
\gamma(t_{2})-\gamma(t_{0})\right\vert <\frac{\varepsilon}{2}+\frac
{\varepsilon}{2}<(d+1)\varepsilon
\end{align*}
Also, since, by (iv), the smaller arc of $C$ from $\sigma(\theta_{2})$ to
$\sigma(\theta_{1})$, equal to the sum of the smaller arc between
$\sigma(\theta_{2})$ and $\sigma(2\pi)$ and that between $\sigma(0)$ and
$\sigma(\theta_{1})$, has length $<d\varepsilon$,%
\[
\left\vert \gamma_{1}(1)-\gamma_{3}(1)\right\vert =\left\vert \sigma
(\theta_{1})-\sigma(\theta_{2})\right\vert <d\varepsilon<(d+1)\varepsilon.
\]
It follows that for $0\leq t\leq1,$%
\begin{align*}
\left\vert \gamma_{1}(t)-\gamma_{3}(t)\right\vert  & \leq\left(  1-t\right)
\left\vert \gamma_{1}(0)-\gamma_{3}(0)\right\vert +t\left\vert \gamma
_{1}(1)-\gamma_{3}(1)\right\vert \\
& <\left(  1-t\right)  (d+1)\varepsilon+t(d+1)\varepsilon\\
& =(d+1)\varepsilon.
\end{align*}
Hence $\left\Vert \gamma_{1}-\gamma_{3}\right\Vert _{\left[  0,1\right]  }%
\leq(d+1)\varepsilon$. On the other hand, by (iv) and elementary geometry,
$\left\vert \gamma_{3}(1)-\gamma_{3}(0)\right\vert =\rho(\gamma(t_{2}%
),C)<\delta$. Now, $\gamma_{3}$ and $\gamma_{1}$ are paths in the compact set
$K\equiv\overline{\mathsf{car}(\gamma_{3})\cup\mathsf{car}(\gamma_{1})}$,
$\rho(z_{1},K)=d/2$, and $z\rightsquigarrow1/(z-z_{1})$ is differentiable on
$K$. Hence, by Lemma\ \ref{jan09l2},%
\begin{align}
\left\vert \int_{\gamma_{1}}\frac{\mathsf{d}z}{z-z_{1}}-\int_{\gamma_{3}}%
\frac{\mathsf{d}z}{z-z_{1}}\right\vert  &  \leq\frac{\left\Vert \gamma
_{1}-\gamma_{3}\right\Vert _{\left[  0,1\right]  }}{d/2}\left(  2+\frac
{\left\vert \gamma_{3}(1)-\gamma_{3}(0)\right\vert }{d/2}\right)  \nonumber\\
&  \leq\frac{2(d+1)\varepsilon}{d}\left(  2+\frac{2\delta}{d}\right)
\nonumber\\
&  =4\left(  1+\frac{1}{d}\right)  \left(  1+\frac{\delta}{d}\right)
\varepsilon.\label{aan1}%
\end{align}
Next we have%
\begin{align}
\left\vert \left(  \int_{\gamma}-\int_{\gamma_{0}}-\int_{\gamma_{4}}\right)
\frac{\mathsf{d}z}{z-z_{1}}\right\vert  &  =\left\vert \int_{t_{1}}^{t_{2}%
}\frac{\gamma^{\prime}(t)}{\gamma(t)-z_{1}}\mathsf{d}t\right\vert \nonumber\\
&  \leq\int_{t_{1}}^{t_{2}}\frac{\left\vert \gamma^{\prime}(t)\right\vert
}{\left\vert \gamma(t)-z_{1}\right\vert }\mathsf{d}t\nonumber\\
&  \leq\int_{t_{1}}^{t_{2}}\frac{M}{d}\mathsf{d}t=\frac{M}{d}(t_{2}%
-t_{1})<\frac{M}{d}2r_{1}<\frac{M\varepsilon}{d}.\label{aan2}%
\end{align}
On the other hand, from (iv) we have%
\begin{align}
\left\vert \left(  \int_{\sigma}-\int_{\gamma_{2}}\right)  \frac{\mathsf{d}%
z}{z-z_{1}}\right\vert  &  =\left\vert \left(  \int_{0}^{\theta_{1}}%
+\int_{\theta_{2}}^{2\pi}\right)  \frac{\sigma^{\prime}(\theta)}{\sigma
(\theta)-z_{1}}\right\vert \nonumber\\
&  \leq\left(  \int_{0}^{\theta_{1}}+\int_{\theta_{2}}^{2\pi}\right)
\left\vert \frac{\sigma^{\prime}(\theta)}{\sigma(\theta)-z_{1}}\right\vert
\mathsf{d}\theta\nonumber\\
&  \leq\left(  \int_{0}^{\theta_{1}}+\int_{\theta_{2}}^{2\pi}\right)
1\mathsf{d}\theta\nonumber\\
&  =\theta_{1}+\left(  2\pi-\theta_{2}\right)  <\varepsilon.\label{aan3}%
\end{align}
Now%
\begin{align*}
& \left(  \int_{\gamma}+\int_{\sigma}\right)  \frac{\mathsf{d}z}{z-z_{1}}\\
& =\left(  \int_{\Gamma}+\left(  \int_{\gamma}-\int_{\gamma_{0}}-\int%
_{\gamma_{4}}\right)  +\left(  \int_{\sigma}-\int_{\gamma_{2}}\right)
-\left(  \int_{\gamma_{1}}-\int_{\gamma_{3}}\right)  \right)  \frac
{\mathsf{d}z}{z-z_{1}}%
\end{align*}
from which, using (\ref{aan15}), (\ref{aan1}), (\ref{aan2}), and (\ref{aan3}),
we obtain
\begin{align*}
\left\vert \left(  \int_{\gamma}+\int_{\sigma}\right)  \frac{\mathsf{d}%
z}{z-z_{1}}\right\vert  &  \leq\left\vert \int_{\Gamma}\frac{\mathsf{d}%
z}{z-z_{1}}\right\vert +\left\vert \left(  \int_{\gamma}-\int_{\gamma_{0}%
}-\int_{\gamma_{4}}\right)  \frac{\mathsf{d}z}{z-z_{1}}\right\vert \\
&  +\left\vert \left(  \int_{\sigma}-\int_{\gamma_{2}}\right)  \frac
{\mathsf{d}z}{z-z_{1}}\right\vert +\left\vert \left(  \int_{\gamma_{1}}%
-\int_{\gamma_{3}}\right)  \frac{\mathsf{d}z}{z-z_{1}}\right\vert \\
&  <0+\frac{M}{d}\varepsilon+\varepsilon+4\left(  1+\frac{1}{d}\right)
\left(  1+\frac{\delta}{d}\right)  \varepsilon\\
&  =\kappa\varepsilon.
\end{align*}
Since $\varepsilon\in\left(  0,\min\left\{  \alpha,\pi\right\}  \right)  $ is
arbitrary and independent of $\kappa$, it follows that%
\[
\left(  \int_{\gamma}+\int_{\sigma}\right)  \frac{\mathsf{d}z}{z-z_{1}}=0.
\]
%

\medskip
\noindent
\textbf{Step 5: }Conclusion of the proof.%

\smallskip
\noindent
From Step 4 we obtain:%
\begin{align*}
\mathsf{j}(\gamma,z_{1}) &  =\frac{1}{2\pi i}\int_{\gamma}\frac{\mathsf{d}%
z}{z-z_{1}}\\
&  =\frac{-1}{2\pi i}\int_{\sigma}\frac{\mathsf{d}z}{z-z_{1}}\\
&  =\frac{-1}{2\pi i}\int_{0}^{2\pi}\frac{\sigma\,^{\prime}(\theta)}%
{\sigma(\theta)-z_{1}}\mathsf{d}\theta\\
&  =\frac{-1}{2\pi i}\int_{0}^{2\pi}-i\mathsf{d}\theta=1=\mathsf{ind}%
(z_{1};\gamma).
\end{align*}
Finally, if $\zeta$ is any point of $\mathsf{in}(\gamma)$, then it can be
joined to $z_{1}$ by a polygonal path that lies in $-\mathsf{car}(\gamma)$, so
by Proposition \ref{oct31p1}, $\mathsf{j}(\gamma,\zeta)=\mathsf{j}%
(\gamma,z_{1})=1=\mathsf{ind}(\zeta;\gamma)$.%
\hfill

\end{proof}

\section{Concluding remarks}

The properties we have established for Jordan curves and their corresponding
indexes prepare the way for application in the constructive theory of zeroes,
poles, and meromorphic functions, to which we shall turn our attention in
papers that are now in preparation.%

\bigskip

%

\bigskip
%

\bigskip
%

\noindent
\textbf{Author's address: \ }School of Mathematics and Statistics, University
of Canterbury, Christchurch 8140, New Zealand%

\noindent
\textbf{email}: dsbridges.math@gmail.com

\vfill

\begin{flushright}
{\small \copyright \ Douglas S. Bridges, 14 February 2025}
\end{flushright}


\begin{thebibliography}{99}                                                                                               %
\bibitem {JCTthm}G. Berg, W. Julian, R. Mines, F. Richman, `The Constructive
Jordan Curve Theorem', Rocky Mountain J. Math. \textbf{5}(2), 1975.

\bibitem {Bishop}E. Bishop, \emph{Foundations of Constructive Analysis,
}McGraw-Hill, New York, 1967.

\bibitem {BB}E. Bishop and D.S. Bridges, \emph{Constructive Analysis},
Grundlehren der math. Wissenschaften \textbf{279}, Springer Verlag,
Heidelberg-Berlin-New York, 1985.

\bibitem {BR}D.S. Bridges and F. Richman, \emph{Varieties of Constructive
Mathematics}, London Math. Soc. Lecture Notes \textbf{97}, Cambridge Univ.
Press, 1987.

\bibitem {BVtech}D.S. Bridges and L.S. V\^{\i}\c{t}\u{a}, \emph{Techniques of
Constructive Analysis}, Universitext, Springer New York, 2006.

\bibitem {Bridges24}D.S. Bridges, `Improving Cauchy's Integral Theorem in
Constructive Analysis', http://arxiv.org/abs/2410.11058

\bibitem {Handbook}D.S. Bridges, H. Ishihara, M.J. Rathjen, H. Schwichtenberg
(editors),\emph{\ Handbook of Constructive Mathematics}, Encyclopedia of
Mathematics and Its Applications \textbf{185}, Cambridge University Press, 2023

\bibitem {Conway}J.B. Conway, \emph{Functions of One Complex Variable}, Grad.
Texts in Math. \textbf{11}, Springer-Verlag, New York, 1973.

\bibitem {Dieudonne}J. Dieudonn\'{e}, \emph{Foundations of Modern Analysis},
Academic Press, New York, 1960.

\bibitem {BillFred84}W.H. Julian and F. Richman, `A uniformly continuous
function on $\left[  0,1\right]  $ that is everywhere different from its
infimum', Pacific J. Math. \textbf{111}(2), 333--340, 1984.
\end{thebibliography}
\end{document}